\documentclass[10pt]{amsart}\oddsidemargin -1mm \evensidemargin -1mm \topmargin -15mm\headheight5mm \headsep 3.5mm \textheight 245mm\textwidth 170mm 
\usepackage{amsfonts, amssymb, amsthm, amsmath, euscript,multicol}

\numberwithin{equation}{section}
\theoremstyle{plain}
\newtheorem{theorem}[equation]{Theorem}
\newtheorem{proposition}[equation]{Proposition}
\newtheorem{corollary}[equation]{Corollary}
\newtheorem{lemma}[equation]{Lemma}

\newtheorem{conjecture}[equation]{Conjecture}

\theoremstyle{definition}
\newtheorem{definition}[equation]{Definition}
\newtheorem{defn}[equation]{Definition}

\newtheorem{remark}[equation]{Remark}

\newcommand{\beq}{\begin{equation}}
\newcommand{\eeq}{\end{equation}}

\newcommand{\ang}[1]{\langle #1 \rangle}

\newcommand{\mb}{\mathbb}
\newcommand{\mc}{\mathcal}
\newcommand{\mf}{\mathfrak}

\newcommand{\kk}{{\Bbbk}}

\newcommand{\sG}{\mc{G}}

\newcommand{\sP}{\mc{P}}

\newcommand{\CC}{\mb{C}}

\newcommand{\NN}{\mb{N}}

\newcommand{\QQ}{\mb{Q}}
\newcommand{\RR}{\mb{R}}

\newcommand{\VV}{\mb{V}}

\newcommand{\ZZ}{\mb{Z}}

\newcommand{\DMO}{\DeclareMathOperator}
\DMO{\GK}{GKdim}
\DMO{\PGK}{PGKdim}
\DMO{\gr}{gr}
\DMO{\Der}{Der}
\DMO{\trdeg}{trdeg}
\DMO{\Kdim}{Kdim}
\DMO{\codim}{codim}
\DMO{\dd}{d}
\DMO{\pd}{pd}
\DMO{\Spec}{Spec} 
\DMO{\Supp}{Supp}
\DMO{\op}{op}
\DMO{\LT}{LT}

\newcommand{\del}{\partial}
\newcommand{\ssm}{\setminus}
\newcommand{\hra}{\hookrightarrow}
\providecommand{\abs}[1]{\lvert#1\rvert}
\newcommand{\Qbar}{\overline{\QQ}}

\renewcommand{\S}{{\rm S}}
\newcommand{\U}{{\rm U}}
\newcommand{\st}{:}  
\DMO{\spn}{span}

\newcommand{\blambda}{\pmb{\lambda}}
\newcommand{\bmu}{\pmb{\mu}}
\newcommand{\bnu}{\pmb{\nu}}

\title{Ideals in the enveloping algebra of the positive Witt algebra}
\author{Alexey V. Petukhov and Susan J. Sierra}
\date{\today}

\address{Sierra: School of Mathematics, The University of Edinburgh, Edinburgh EH9 3FD, United Kingdom}

\email{s.sierra@ed.ac.uk}

\address{Petukhov: Jacobs University Bremen, Bremen 28759, Germany\\on leave from Institute for Information Transmission problems, Moscow 127051, Russia}
\email{alex-{}-2@yandex.ru}


\begin{document}

\begin{abstract}
Let $W_+$ be the {\em positive Witt algebra}, which has a $\CC$-basis $\{e_n: n \in \ZZ_{\geq 1}\}$, with Lie bracket $[ e_i, e_j] = (j-i) e_{i+j}$. 
We study the two-sided ideal structure of the universal enveloping algebra $\U(W_+)$ of $W_+$. 
We show that if $I$ is a (two-sided) ideal of $\U(W_+)$ generated by quadratic expressions in the $e_i$, then $\U(W_+)/I$ has finite Gelfand-Kirillov dimension, and that such ideals satisfy the ascending chain condition.  
We conjecture that  analogous facts hold for arbitrary ideals of $\U(W_+)$, and verify a version of these conjectures for radical Poisson ideals of the symmetric algebra ${\rm S}(W_+)$.
\end{abstract}

\maketitle

\section{Introduction}\label{INTRO}
Let $\kk$ be a 
field of characteristic zero, and let $W_+$ be the {\em positive Witt algebra}, which has a $\kk$-basis $$\{e_n: n \in \ZZ_{\geq 1}\},$$ with Lie bracket 
\beq\label{Witt}
[ e_i, e_j] = (j-i) e_{i+j}. 
\eeq
This paper studies the two-sided ideal structure of $\U(W_+)$.

In 2013, the second author and Walton proved \cite{SW1} that $\U(W_+)$ is neither left nor right noetherian, by establishing the analogous properties for the quotient ring $B = \U(W_+)/ ( e_1e_5 - 4e_2e_4 + 3e^2_3 + 2e_6)$.  
However,  by \cite[Proposition~6.6]{SW2}, two-sided ideals of $B$ satisfy the ascending chain condition, and  $B$ has Gelfand-Kirillov dimension (GK-dimension) 3.
The main question this paper investigates is how far these properties generalise to arbitrary quotients of $\U(W_+)$.

The enveloping algebra $\U(W_+)$ is highly noncommutative --- it is well-known, for example, that the Weyl algebra $A_n(\kk)$ is a quotient of $\U(W_+)$ for any $n$.
(This can be seen by combining Theorem~4.7.9 and Section 6.2 of \cite{Dix} with the observation that if $n> 1$, factoring out the Lie ideal generated by $e_n$ gives a finite-dimensional Lie algebra of nilpotency class $n-2$.)

One thus  expects that two-sided ideals of $\U(W_+)$ are large, and computer experiments have supported this.  
In fact, all known proper quotients of $\U(W_+)$ have finite GK-dimension, even though $\U(W_+)$ has subexponential growth and thus infinite GK-dimension.
We conjecture:
\begin{conjecture}\label{conj1}
The enveloping algebra $\U(W_+)$ has {\em just infinite GK-dimension} in the sense that if $I$ is a nonzero ideal of $\U(W_+)$, then the GK-dimension of $\U(W_+)/I$ is finite.
\end{conjecture}

If nontrivial ideals in $\U(W_+)$ are large, it is natural to expect that the lattice of two-sided ideals is well-behaved.  In fact, we conjecture:
\begin{conjecture}\label{conj2}
Two-sided ideals of $\U(W_+)$ satisfy the ascending chain condition:  all strictly ascending chains of ideals are finite.
\end{conjecture}
The second conjecture,  asked in \cite[Question~0.11]{SW2}, was first brought to the second author's attention by Lance Small. 

The first author and Penkov have  shown that the ideal structure of enveloping algebras of infinite-dimensional Lie algebras can be extremely sparse; for example, for the majority of locally simple Lie algebras $\mathfrak g_\infty$, the universal enveloping algebra $\U(\mathfrak g_\infty)$ has only finitely many two-sided ideals by  \cite[Corollary 3.2 and Section~6]{PP1}. 
Further, the analogue of Conjecture~\ref{conj2} holds for $\U(\mathfrak{sl}(\infty))$ by \cite[Corollary~5.4]{PP2}.
In general,  two-sided ideals of enveloping algebras of infinite-dimensional Lie algebras form an interesting area of research with many  unexpected phenomena.

Although we do not prove either conjecture, we make progress towards both, establishing several partial results that support the conjectures.
Our key method is to work with the symmetric algebra $\S(W_+)$ under the natural Poisson structure induced from $\U(W_+)$. 
It is well-known that ideals of $\U(W_+)$ give rise, via the associated graded construction, to Poisson ideals of $\S(W_+)$.
We show (Lemma~\ref{Lns}) that if $I$ is a nontrivial radical Poisson ideal of $\S(W_+)$ then $\S(W_+)/I$ embeds in a finitely generated commutative algebra.  As a consequence, we obtain:
\begin{theorem}\label{ithm1}
(Corollary~\ref{Cns}, Corollary~\ref{cor1})
Let $K$ be a nontrivial Poisson ideal of $\S(W_+)$.
Then $K$ has finitely many minimal primes, and $\S(W_+)/K$ has finite GK-dimension.
\end{theorem}

Using this result, we show:
\begin{theorem}\label{ithm2}
(Theorem~\ref{thm:acc})
The algebra $\S(W_+)$ satisfies the ascending chain condition on radical Poisson ideals.
\end{theorem}
It follows that $\U(W_+)$ satisfies the ascending chain condition on ideals whose associated graded ideal is radical, see Corollary~\ref{cor:acc}.

We then turn to studying the GK-dimension of quotients of $\U(W_+)$ more directly.
 For a Poisson algebra $A$, we define the  {\em Poisson Gelfand-Kirillov dimension} $\PGK A$, which measures the growth of $A$ as a Poisson algebra.  
 We show (Theorem~\ref{thm:GKW2}) that the GK-dimension of a quotient $R$ of $\U(W_+)$ is equal to the Poisson GK-dimension of the associated quotient of $\S(W_+)$.  \footnote{Since $\S(W_+)$ is not finitely generated as an algebra, there is no clear reason for the  GK-dimension of the associated quotient of $\S(W_+)$ to give a bound on the GK-dimension of $R$ in general.}
 We further show:
 \begin{theorem}\label{ithm3}
 If $K$ is a nontrivial radical Poisson ideal of $\S(W_+)$, then 
 \[\PGK \S(W_+)/K = \GK \S(W_+)/K,\]
  which we have seen previously is finite.
 \end{theorem}
Therefore, if $I$ is an ideal of $\U(W_+)$ whose associated graded ideal is radical, then $\GK \U(W_+)/I < \infty$, and thus Conjectures~\ref{conj1} and \ref{conj2} both hold for ideals whose associated graded ideal is radical.
 
 We then turn our attention to quadratic elements in the symmetric algebra, i.e. elements of $\S^2(W_+)$.
 Through explicit computations, we show that $ \S^2(W_+)$ is a noetherian $W_+$-module (Theorem~\ref{thm:S2}), and as a consequence that $\S(W_+)$ satisfies the ascending chain condition on Poisson ideals generated by quadratic elements.  
 Finally, we show:
 
 \begin{theorem}\label{ithm4}
 (Corollary~\ref{cor:GKsofar})
 If $I$ is an ideal of $\U(W_+)$ that contains a quadratic expression in the $e_i$, then $\U(W_+)/I$ has finite GK-dimension.
 \end{theorem}
 
 Recall that $W_+$ is a subalgebra of the (full) {\em  Witt algebra} $W$, which has a $\kk$-basis $\{ e_n \st n \in \ZZ\}$ and Lie bracket defined by \eqref{Witt}.  
 Recall also that $W$ is obtained from the {\em Virasoro algebra} $V$ (which we do not define) by setting the central charge equal to zero.
 We conjecture that analogues of Conjectures~\ref{conj1} and \ref{conj2} and Theorem~\ref{ithm4} hold for $\U(W)$ and $\U(V)$.  
 These questions will be the subject of future work.
 
 The organisation of the paper is as follows.  Section~\ref{RADICALS}, where   we prove Theorems~\ref{ithm1} and \ref{ithm2}, focuses on quotients of $\S(W_+)$ by radical Poisson ideals.
 In Section~\ref{DIMENSION} we define the Poisson Gelfand-Kirillov dimension of a Poisson algebra, give some of its properties, and prove Theorem~\ref{ithm3}.
 In Section~\ref{QUADRATIC} 	we study the structure of $\S^2(W_+)$ and prove Theorem~\ref{ithm4}.
 This proof involves  computer calculations which are discussed in an appendix.

{\bf Acknowledgements: }  
The first author was supported by Leverhulme Trust Grant RPG-2013-293 and RFBR grant 16-01-00818. The second author was supported by EPSRC grant EP/M008460/1.

We would like to thank  Jacques Alev, Tom Lenagan, Omar Leon Sanchez, Paul Smith and Toby Stafford for helpful discussions.  We would particularly like to thank Ioan Stanciu, whose computer experiments, done as part of his MMath dissertation at the University of Edinburgh, gave us experimental evidence for Conjecture~\ref{conj1}.

\section{Poisson ideals}\label{RADICALS}
We begin by collecting some basic properties of Poisson algebras, and then move to deriving consequences for $\S(W_+)$.  We note that all Poisson algebras in this paper are commutative as algebras.

Our convention is that  $\NN$ is equal to the set of nonnegative integers, and $\ZZ_{\geq 1}$ is the set of positive integers.

\subsection{Operations on ideals} 
Since we will be working with the non-noetherian ring $\S(W_+) \cong \kk[x_1, x_2, \dots]$, we recall some basic concepts in commutative algebra which do not depend on the ascending chain condition.  

Throughout the next two  subsections $A$ is a Poisson $\kk$-algebra, $I$ is a Poisson ideal of $A$, and $a, b, c$ are elements of $A$. 

Recall that
$$(I: b):=\{a\in A\st ab\in I\},$$
and note $I \subseteq (I:b)$.
Also recall that an ideal $I$ is {\it radical} if $I=\sqrt I :=\{a\in A\st a^n\in I \mbox{ for some } n \in \NN\}$.

Define
$$(I: b^\infty):=\{a\in A\st  ab^n\in I  \mbox{ for some } n \in \NN\},$$
$$ (I\hat+ b):=\bigcap_{a\in (I: b^\infty)}(I: a^\infty).$$

\begin{lemma}\label{lem:rad1}
If $I$ is a radical  ideal then $(I:b)$ and  $(I\hat+ b)$ are  radical  for any $b\in A$.  Further,
$$(I:b)=(I: b^\infty),$$
and $b \in (I\hat+b)$.
\end{lemma}
\begin{proof}
First we show that $(I: b)=(I: b^\infty)$. It is clear that $(I: b)\subseteq (I: b^\infty)$. 
Thus it is enough to show that $(I: b^\infty)\subseteq (I: b)$. We fix $a\in (I: b^\infty)$ and $n\in \NN$ such that $ab^n\in I$. We  have that $(ab)^n=a^{n-1}ab^n\in I$. Hence $ab\in I$ and $a \in (I:b)$.

Next, we wish to show that $\sqrt{(I: b)}=(I:b)$. We fix $a\in \sqrt{(I: b)}$ and $n\in \NN$ such that $a^n\in (I: b)$. 
We have   $(ab)^n=a^nbb^{n-1}\in I$ and therefore $ab\in I$. Hence $a\in (I:b)$.

An intersection of a collection of radical ideals is clearly a radical ideal and thus if $I$ is radical so is $(I\hat + b)$.

For the final statement, if $a \in (I: b^\infty) = (I:b)$ then $ab \in I$ and $b \in (I:a^\infty)$.  
\end{proof}

If $b \in A$ then we define $A[b^{-1}] =  A[x]/(xb-1)$, where we denote $x$ by $b^{-1}$.
The kernel of the natural map $A \to A[b^{-1}] $ is $((0): b^\infty)$.
Likewise, the kernel of the natural map $A \to (A/I)[b^{-1}]$ is $(I: b^\infty)$.
We then have:

\begin{lemma}\label{Lredrad}
If $I$ is a radical ideal of $A$, then  $(A/I)[b^{-1}]$ is  reduced  for any $b\in A$ (i.e. $(A/I)[b^{-1}]$ has no nonzero nilpotents).\end{lemma}
\begin{proof} 
It suffices to consider the case $I = (0)$.
If $(ab^{-n})^k = 0$ in $A[b^{-1}]$ then the natural map $A \to A[b^{-1}]$ sends $a^k \mapsto 0$ and so 
$a^k \in ((0): b^\infty)$. 
By Lemma~\ref{lem:rad1} we have $a \in ((0): b^\infty) $.
Thus $a b^{-n} =0$ in $A[b^{-1}]$.
\end{proof}
\begin{lemma} \label{lem:rad2}
We  have $(I\hat+ b)\cap (I:b)\subseteq\sqrt I$ and thus  $\sqrt I=\sqrt{(I\hat+ b)\cap (I:b)} = \sqrt{(I\hat+ b)} \cap \sqrt{(I:b)}$.
\end{lemma}
\begin{proof} 
Let $a\in (I\hat+ b)\cap (I:b)$. Then $a\in (I:b^\infty)$ and therefore from the definition of $(I \hat+ b)$ we have $a\in (I:a^\infty)$. Hence $a\in\sqrt I$.
The final statement holds since $$I \subseteq (I\hat+b) \cap (I:b).$$
\end{proof}

Although the Lasker-Noether   primary decomposition theorem does not hold if $A$ is not noetherian, Lemma~\ref{lem:rad2} can provide a useful analogue.  

\subsection{Compatibility with Poisson structure}
We now show that the constructions above preserve the Poisson structure of $A$.
\begin{lemma}\label{lem:comp1}
 If $I$ is a Poisson ideal  of A then so is $(I: b^{\infty})$.\end{lemma}
\begin{proof} Fix $a\in A$ and $n\in\NN$ such that $ab^n\in I$. It is enough to show that for any $c\in A$ we have $\{a, c\}\in (I: b^{\infty})$. We have
$$\{ab^{n+1}, c\}=\{a, c\}b^{n+1}+(n+1)ab^n\{b,c\}.$$ 
The terms $\{ab^{n+1}, c\}$ and $(n+1)ab^n\{b,c\}$ belong to $I$ and thus $\{a, c\}b^{n+1}\in I$.
\end{proof}

We immediately obtain:
\begin{corollary}\label{cor:comp2}
If $I$ is a Poisson ideal then the algebra $(A/I)[b^{-1}]$ is Poisson with respect to the Poisson bracket defined as follows:
\[ \{a_1 b^{-n_1}, a_2 b^{-n_2} \} = \left( \{a_1, a_2\} b - n_2  \{a_1, b\} a_2 - n_1 \{b, a_2\} a_1 \right) b^{-n_1-n_2 -1}.\]
The natural maps $A\to A/I\to (A/I)[b^{-1}]$ are morphisms of Poisson algebras.
\qed
\end{corollary}

Corollary~\ref{cor:comp2} is a special case of a more general result:  that if $A$ is a Poisson algebra and $\mc{C}$ is a multiplicatively closed set in $A$ then $A \mc{C}^{-1}$ has a natural Poisson structure compatible with that on $A$.  

Let $P$ be a minimal prime of the commutative algebra $A$, let $\mc{C} = A \ssm P$, and let 
\[ Q = \{ x\in A \st xc =0 \mbox{ for some $c \in \mc{C}$ } \}\]
be the kernel of the natural map $A \to A \mc{C}^{-1}$.
If $xyc = 0$ where $y, c \in \mc{C}$, then $yc \in \mc{C}$ and so $x \in Q$.
Thus if $xy \in Q$ and $y \not \in P$, then $x\in Q$.  However, even if $A$ is a quotient of $\S(W_+)$, we do not know if   $Q$ must be primary.
Note that if $A$ is in addition noetherian, then $P \mc{C}^{-1}$ is the unique minimal prime of the noetherian ring $A \mc{C}^{-1}$ and so is nilpotent.  Thus if $x\in P$, we have some $x^n \in Q$ and $Q$ is in addition $P$-primary.  
%

\begin{lemma}[{see also~\cite[Lemma 1.8]{DiffAlg}}]\label{Lprad} 
If $I$ is Poisson then $\sqrt I$ is Poisson.
\end{lemma}
\begin{proof}We fix $a, b\in A$ and $n\in\NN$ such that $a^n\in I$. It is enough to show that $\{a, b\}\in \sqrt I$. We will prove that $1\in (I: \{a, b\}^{\infty})$ (this statement is equivalent to the previous one). Assume to the contrary that $1\not\in(I: \{a, b\}^{\infty})$.

We have that $a^n\in I\subseteq (I:\{a, b\}^\infty)$. Let $m$ be the minimal nonnegative integer such that $a^m\in (I:\{a, b\}^\infty)$. 
Since  $1\not\in (I:\{a, b\}^\infty)$, thus $m\ge1$. 
 As by Lemma~\ref{lem:comp1} $(I: \{a,b\}^\infty)$ is Poisson, 
$$ma^{m-1}\{a, b\}=\{a^m, b\}\in (I:\{a, b\}^\infty).$$
Therefore $a^{m-1}\in(I:\{a, b\}^\infty)$. This contradicts our assumption on the minimality of $m$.
\end{proof}

We thus obtain: 
\begin{corollary}\label{Cdec}
If $I$ is a radical Poisson ideal then for any $b\in A$ both $(I: b)$ and $(I\hat + b)$ are radical Poisson ideals and $I=(I:b)\cap (I\hat+b)$.\end{corollary}
\begin{proof}
Combine Lemmas~\ref{lem:rad1}, \ref{lem:rad2}, and \ref{lem:comp1}.
\end{proof}

It is well known \cite[Corollary~2.12]{Eisenbud} that any radical ideal $I$ of $A$ is an intersection of prime ideals and thus of primes minimal over $I$ --- this follows from Zorn's Lemma and does not require $A$ to be noetherian.
If $I$ has finitely many minimal primes $\mf{p}_1, \dots, \mf{p}_m$ then $\sqrt I = \mf{p}_1 \cap \dots \cap \mf{p}_m$.  
Conversely, if $\sqrt I = \mf{p}_1 \cap \dots \cap \mf{p}_m$ is an irredundant intersection then the $\mf{p}_j$ are precisely the minimal primes of $I$, as if $I \subseteq \mf{q}$ for some prime $\mf{q}$ then some $\mf{p}_j \subseteq \mf{q}$.

\begin{lemma}\label{lem:minprime}
Let $I$ be a  Poisson ideal of $A$. The minimal primes of $I$ are Poisson ideals. 
\end{lemma}
\begin{proof} 
Without loss of generality $I = \sqrt I$ is radical.
Let $\mf{p}$ be a minimal prime over $I$. Let $I_{\mf{p}}$ be the sum of all Poisson ideals contained in $\mf p$. 
Clearly $I_{\mf p} $ is  the maximal Poisson ideal contained in $\mf{p}$. To complete the proof it is enough to show that $I_{\mf{p}}$ is prime.

Certainly $\sqrt {I_{\mf{p}}}\subseteq \mf{p}$.  Since $\sqrt I_{\mf{p}}$ is Poisson by Lemma~\ref{Lprad}, $I_{\mf{p}}$ is a radical ideal. Let $x, y\in A$ be such that $xy\in I_{\mf{p}}$. 
We will show that either $x\in I_{\mf{p}}$ or $y\in I_{\mf{p}}$. 
By definition, $y\in (I:x)$, and by Lemma~\ref{lem:rad1},  $x\in (I\hat+ x)$. 
By Corollary~\ref{Cdec}, $I=(I:x)\cap (I\hat+ x)$, and both $(I:x)$ and $ (I\hat+ x)$ are Poisson ideals. Since  $I\subseteq \mf{p}$, either $(I:x)\subseteq \mf{p}$ or $(I\hat+ x)\subseteq \mf{p}$. 
Thus  either \begin{center}$y\in (I:x)\subseteq I_{\mf p}$ or $x\in(I\hat+x) \subseteq I_{\mf p}$.\end{center}
\end{proof}

\subsection{Radical ideals in $\S(W_+)$}\label{SWRADICAL}

The {\em positive Witt algebra}  is the Lie algebra $W_+$ with basis $e_i~(i\in\ZZ_{\geq 1})$ and Lie bracket $[e_i, e_j]=(j-i) e_{i+j}$. 
The symmetric algebra of $W_+$ is denoted by $\S(W_+)$.
Our convention is that the image of $e_i $ in $  {\S}(W_+)$ is denoted by $x_i$. 

We now specialise to studying the Poisson structure on $\S(W_+)$ induced by the Lie bracket on $W_+$.
In this section, we will show that $\S(W_+)$ satisfies the ascending chain condition on radical Poisson ideals and that proper quotients by Poisson ideals have finite Gelfand-Kirillov dimension. 
Our first step is to show that any nontrivial quotient of $\S(W_+)$ by a radical ideal embeds into a finitely generated Poisson algebra.

As with any symmetric algebra, $\S(W_+)$ carries a natural grading, which we refer to the {\em order gradation} and denote by $o$.  
We have $o(x_i) = 1$ for all $i$, and $o(\{x_i, f\}) \leq o(f)$ for all $i$ and  for all $f\in \S(W_+)$.
On $\U(W_+)$, there is an order {\em filtration}, which we also denote by $o$, with $o(e_i) =1 $ for all $i$.
Recall that $\S(W_+) = \gr_o \U(W_+)$ is the associated graded ring of the order filtration on $\U(W_+)$.

In addition, $W_+$ is a graded Lie algebra if we give $e_i$ degree $i$, and this extends to  a graded structures on $\U(W_+)$ and $\S(W_+)$, which we refer to as the {\em degree gradation}.  
We denote the degree gradation by $d$, with  $d(e_i) = d(x_i) = i$.

\begin{lemma}\label{Lns} Let $I$ be a  nontrivial 
radical Poisson ideal of ${\S}(W_+)$.
\begin{itemize}
\item[(a)]  There exists a finitely generated reduced commutative algebra $A$ such that
 there is an embedding \[{\S}(W_+)/I \hra A.\]
\item[(b)] If $I$ is prime, then there exists a finitely generated subalgebra $B$ of $\S(W_+)/I$ and $p \in B$ so that $\S(W_+)/I \subseteq B[p^{-1}]$.
\item[(c)] The algebras $A$ and $B[p^{-1}]$ in parts (a) and (b) carry natural Poisson structures compatible with that of $\S(W_+)/I$.
\end{itemize}
\end{lemma}
To prove Lemma~\ref{Lns} we need several auxiliary facts.
Let $f\in I$ be a  nonzero element of minimal   order.
 We pick the smallest number $n$ such that $f \in \kk[x_1, \dots, x_n]$. 
 The following lemma is straightforward.
\begin{lemma} \label{lem:straightforward}
Let $o = o(f)$.
We have $\{x_1, f \}=x_{n+1}p(x_n, x_{n-1}\dots)+q(x_n, x_{n-1},\dots)$ where $p$ and $q$ are  polynomials of order respectively  $\leq o-1$ and $\leq o$.
Further, for $t\in{\ZZ_{\geq 1}}$  we have $$\{x_t, \{x_1, f\}\}=(n+1-t)x_{n+1+t}p+q'(x_{n+t}, x_{n+t-1},\dots)$$ where $q'$ is a  polynomial of order $ \leq o$.
\qed
\end{lemma}

We now prove Lemma~\ref{Lns}.
\begin{proof}[Proof of Lemma~\ref{Lns}]
Let $f$ be as before. If $o(f) = 0$ than $f$ is constant and therefore $I=\U(W_+)$. Thus we can assume from now on that $o(f)\ge1$.
Assume that the order of $f$ is minimal among all elements of ideals for which the statement of the lemma fails, and let $I \ni f$ be such an ideal. 

Let $p$ be as in the statement of  Lemma~\ref{lem:straightforward}.
Then  $o(p)<o(f)$ and thus that $p\not\in I$. 
By Lemma~\ref{lem:rad1}, $p\in (I\hat+p)$. 
By Corollary~\ref{Cdec} there is  an injective map of Poisson algebras
$${\S}(W_+)/I\to{\S}(W_+)/(I: p)\oplus{\S} (W_+)/(I\hat+ p).$$

By minimality of the degree of $f$ the statement of Lemma~\ref{Lns} holds  for the ideal $(I\hat + p)$. 
Therefore to prove (a) it is enough to show that there is an embedding of ${\S}(W_+)/(I: p)$ into a reduced finitely generated commutative algebra. Consider the natural embedding
$$\phi: {\S}(W_+)/(I: p)\to ({\S}(W_+)/I)[p^{-1}],$$
which is a homomorphism of Poisson algebras by Corollary~\ref{cor:comp2}.
The algebra $({\S}(W_+)/I)[p^{-1}]$ is reduced by Lemma~\ref{Lredrad}. 
Further, by 
Lemma~\ref{lem:straightforward} $({\S}(W_+)/I)[p^{-1}]$ is generated as an algebra by $p^{-1}$ and the images of $$x_1, \dots, x_{2n+2}.$$
 This proves part (a).

We now prove part (b).  Let  $p, x_1,\dots, x_{2n+2}$ be as in the proof of part (a). 
Primality of $I$ implies that $(I:p)=I$ so the natural map $\phi: \S(W_+)/I \to \S(W_+)/I[p^{-1}]$ is injective.
Let $B$ be the subalgebra of ${\S}(W_+)/I$ generated by $x_1,\dots, x_{2n+2}$. 
 It is easy to check that $ p, B$ as above satisfy the conclusions of part (b).
 
 Since the maps involved are homomorphisms of Poisson algebras, (c) also holds.
 \end{proof}

Lemma~\ref{Lns} has the following important consequence:

\begin{corollary}\label{Cns} Let $I$ be a  Poisson ideal of ${\S}(W_+)$. Then $I$ has finitely many minimal primes:  that is, there exist prime ideals $\mf p_1,\dots, \mf p_n$ of ${\S}(W_+)$ such that $\sqrt I=\mf p_1\cap\dots\cap \mf p_n$.
Further, the $\mf{p}_i$ are Poisson ideals.
\end{corollary}
\begin{proof}Thanks to Lemma~\ref{Lns} either $I=(0)$ or there is  an embedding  $$\phi: {\S}(W_+)/\sqrt I\to A$$ of $\S(W_+)/\sqrt I$ into  a reduced finitely generated commutative algebra $A$. For such an algebra $A$ we have
$$\mf p_1\cap\dots \cap \mf p_{n}=(0)$$
for some finite set of prime ideals $\mf p_1,\dots, \mf p_{n}$ of $A$. The ideals $\phi^{-1}(\mf p_1),\dots, \phi^{-1}(\mf p_{n})$ are prime in ${\S}(W_+)$ and we have the desired equality
$$I=\bigcap_{i}\phi^{-1}(\mf p_i).$$
The last sentence is Lemma~\ref{lem:minprime}.
\end{proof}

We wish to show that an ascending chain of radical Poisson ideals in $\S(W_+)$ stabilises.  To do this, we recall two definitions of dimension.  The {\em Krull dimension} of a commutative ring $A$, which we write $\Kdim A$, is the supremum over all strictly ascending chains of prime ideals of $A$ of  the length of the chain minus one. 

The {\em Gelfand-Kirillov dimension} (or {\em GK-dimension})  of $A$ is written $\GK A$ and defined as the supremum over all finite-dimensional subspaces $V$ of $A$ of 
$\varlimsup \log_n \dim_{\kk} V^{n}$ (see also Section~\ref{DIMENSION} and~\cite{KL}).

The following facts are well known.
\begin{proposition} \label{Ldpr}  Let $A$ be a commutative $\kk$-algebra.
\begin{itemize}
\item[(a)]  ${\rm GKdim} A\ge \Kdim A$.
\item[(b)] If $A$ is  finitely generated  then $\operatorname{GKdim} A=\Kdim A$.
\item[(c)] If $A$ is  a subalgebra of a finitely generated commutative $A'$ with $\Kdim A' = \Kdim A$ then 
$$\Kdim A'=\GK A'=\Kdim A=\GK A.$$
\item[(d)] Assume that $A$ is a finitely generated domain and let $p\in A\backslash 0$.
 Let $A'$ be an algebra with $$A \subseteq A' \subseteq A[p^{-1}].$$ 
 Then
$\Kdim A'=\Kdim A$.
\item[(e)]  ${\rm GKdim} (A)= {\rm GKdim} (A/\sqrt {(0)})$.
\end{itemize}
\end{proposition}
\begin{proof}
(a) is \cite[Proposition~3.16]{KL} and (b) is \cite[Theorem~4.5]{KL}.
For (c), we have $$\Kdim A \leq \GK A \leq \GK A' = \Kdim A' = \Kdim A.$$  

Part (d) follows from 
the fact that
\[ \Kdim A \geq \Kdim A' \geq \Kdim A[p^{-1}] = \Kdim A.\]

For (e), note that ${\rm GKdim} A$ is equal to the maximum of the Gelfand-Kirillov dimensions of the finitely generated subalgebras of $A$.
It is immediate from the definitions that $$\Kdim A = \Kdim A/ \sqrt{(0)}.$$   
Therefore (e) follows from (b).
\end{proof}

We now derive some more consequences of Lemma~\ref{Lns}.

\begin{corollary} \label{cor1}
Let $I$ be a nontrivial Poisson ideal of ${\S}(W_+)$. Then the Gelfand-Kirillov dimension of ${\S}(W_+)/I$ is finite.\end{corollary}
\begin{proof} 
By Proposition~\ref{Ldpr}(e),  
we can assume that $I=\sqrt I$. Next, according to Lemma~\ref{Lns} there is an embedding of ${\S}(W_+)/I$ into a  finitely-generated commutative algebra $A$. All such algebras have finite Gelfand-Kirillov dimension which does not exceed the cardinality of a set of generators. Thus
$$\operatorname{GKdim}{\S}(W_+)/I\le\operatorname{GKdim}A< \infty.$$
\end{proof}

\begin{corollary}\label{cor2}
 Let $I$ be a nontrivial Poisson ideal of ${\S}(W_+)$.  Then $$\operatorname{Kdim}({\S}(W_+)/I)=\operatorname{GKdim}({\S}(W_+)/I)< \infty.$$ 
 \end{corollary}
\begin{proof}
As $\Kdim \S(W_+)/I = \Kdim \S(W_+)/\sqrt I$ we may without loss of generality assume that $I = \sqrt I$.
By Corollary~\ref{Cns}, there are Poisson primes $\mf p_1, \dots, \mf p_m$ so that $I = \mf p_1 \cap \dots \cap \mf p_m$.
Thus we have
\[ \max_i \Kdim \S(W_+)/\mf p_i =\Kdim \S(W_+)/I \leq \GK \S(W_+)/I \leq \max_i \GK \S(W_+)/\mf p_i,\]
where the first inequality is Proposition~\ref{Ldpr}(a) and the second comes from the embedding
\[ \S(W_+)/I \hra \bigoplus_i \S(W_+)/\mf p_i,\]
together with \cite[Proposition~3.2]{KL}.
By Lemma~\ref{Lns}(b) and Proposition~\ref{Ldpr}(c,d), for all $i$ we have $$\Kdim \S(W_+)/\mf p_i = \GK \S(W_+)/\mf p_i.$$
Applying Corollary~\ref{cor1}, the result follows.
\end{proof}

The final theorem of this section is:

\begin{theorem}\label{thm:acc} Any ascending sequence of radical Poisson ideals of $\S(W_+)$ stabilises.  
\end{theorem}
The proof of Theorem~\ref{thm:acc} is based on the above results and the following  lemma.
\begin{lemma}\label{lemma:stab}Let $A$ be a commutative algebra and $(0)\subseteq I_1\subseteq I_2\subseteq\dots$ be a chain of radical ideals of $A$ such that:
\begin{itemize}
\item[(a)] $\Kdim (A/I_1)<\infty$,

\item[(b)] there are only finitely many minimal primes over $I_j$ in $A$ for all $j\ge1$.\end{itemize}
Then $I_{j+1}=I_j$ for $j\gg0$, i.e. the sequence $I_1, I_2,\dots$ stabilises.
\end{lemma}
\begin{proof}
If $I$ is an ideal of $A$, write $$\codim I:= \Kdim (A/I).$$
Put $$I_j = P_{j,1} \cap \cdots \cap P_{j,n_j},$$ where the $P_{j,i}$ are the finitely many minimal primes over $I_j$. We have
$$\codim I_j = \max_i \codim P_{j,i}\le\codim I_1<\infty.$$

We induct on $\codim I_1$.  If $\codim I_1 = 0$ then the $P_{j,i}$ are maximal ideals. Since (by primality) each $P_{j+1,i} \supseteq I_j$ contains some $P_{j, i'}$, we have $P_{j+1,i} = P_{j, i'}$ and so $\{ P_{j,*} \}\supseteq \{P_{j+1,*}\}$ and  $n_{j+1} \leq n_j$.  For $j \gg 0$ all $n_j$ are equal and thus all $I_j$ are equal.

So now assume that any ascending chain that begins with a radical Poisson ideal of codimension $k$ must be finite, and suppose that $\codim I_1 = k+1$.
Without loss of generality, all $I_j$ have codimension $k+1$.
Reorder the $P_{j,i}$ so that they have codimension $k+1$ for $i \leq \ell_j$ and codimension $\leq k$ for $\ell_j < i \leq n_j$.
Now  each $P_{j+1,i}$ contains some $P_{j, i'}$ and for dimension reasons if $i \leq \ell_{j+1}$ then we must have $i' \leq \ell_j$ and $P_{j+1,i} = P_{j,i'}$.  Thus $ \ell_{j+1} \leq \ell_j$ and we may assume without loss of generality that all $\ell_j$ are equal to some $\ell$ and for $i \leq \ell $ that all $P_{j,i}$ are equal.

Let $J = P_{1,1} \cap \dots \cap P_{1,\ell}$ and  $K_j = \cap_{i > \ell} P_{j,i}$, so  $I_j = J \cap K_j$.
As all the $P_{j,i}$ are minimal over $I_j$, for fixed $j$ the $P_{j,i}$ are mutually incomparable (i.e. $P_{j, i_1}\not\subseteq P_{j, i_2}$ if $i_1\ne i_2$).
Let $\ell < i \leq n_{j+1}$.  
By primality, $P_{j+1,i} $ does not contain $J$.  
As $P_{j+1,i} \supseteq I_j =J \cap K_j$, we have $P_{j+1,i} \supseteq K_j$ and thus $K_{j+1} \supseteq K_j$.
Since $\codim K_j \leq k$, by induction the $K_j$ stabilise and thus the chain $I_j =  J \cap K_j$ stabilises.
\end{proof}
\begin{proof}[Proof of Theorem~\ref{thm:acc}] Any ascending chain of radical Poisson ideals of $\S(W_+)$ satisfies the assumptions of Lemma~\ref{lemma:stab} thanks to Corollary~\ref{Cns} and Corollary~\ref{cor2}. Therefore Theorem~\ref{thm:acc} follows from Lemma~\ref{lemma:stab}.\end{proof}

Let $J$ be an ideal of $\U(W_+)$.  
Since $fg -gf \in J$ for any $g \in J$, $f\in \U(W_+)$, then $\gr_o(J)$ is a Poisson ideal of $\S(W_+)$.
By \cite[Proposition~1.6.8]{MR}, if $\S(W_+)$ satisfies the ascending chain condition (ACC) on Poisson ideals, it would follow that $\U(W_+)$ has ACC on ideals.   
We cannot prove this at the moment, but the argument above does give:

\begin{corollary}\label{cor:acc}
The algebra $\U(W_+)$ satisfies the ascending chain condition on ideals whose associated graded ideals are radical. \qed
\end{corollary}

We do not know what conditions on an ideal $I$ of $\U(W_+)$ guarantee that the associated graded ideal is radical.
However, it is known that if $I$ is the kernel of one of the homomorphisms from $\U(W_+)$ to an Artin-Schelter regular algebra considered in \cite{SW2}, then the associated graded ideal of $I$ with respect to the order filtration is prime.  Note in this case that $I$ is completely prime.

\begin{remark}\label{Rdalg}
Some results of this section
 can also be deduced from differential algebra (see \cite{DiffAlg} and \cite{Marker}).
Differential algebra (as a branch of mathematics) considers commutative algebras with derivation(s) and the ideals of such algebras which are stable under the derivation(s).  Now, the adjoint action of $W_+$ on itself defines an action of $W_+$ on $\S(W_+)$ by derivations such that
$$e_i\cdot x_j=(j-i)x_{i+j}.$$
The Poisson ideals of $\S(W_+)$ are the ideals of $\S(W_+)$ which are stable under all of the above derivations (equivalently under the derivations induced by $e_1$ and $e_2$).
Thus it is quite natural to connect  results on  the Poisson structure of $\S(W_+)$  with the results of differential algebra.

Consider $\S(W_+)$ as a differential algebra with respect to the derivation $\del_1$ defined by $e_1$.  By the above, any Poisson ideal $I$ of $\S(W_+)$ is a differential ideal with respect to $\del_1$. It is easy to check that $(\S(W_+), \del_1)$ is generated by $x_1, x_2$ as a differential algebra. It follows from the Ritt-Raudenbush basis theorem~\cite[Theorem~7.1]{DiffAlg} that any chain of radical $\del_1$-differential ideals of $\S(W_+)$ stabilises, and thus any chain of radical Poisson ideals of $\S(W_+)$ stabilises. 

Note also that it can be deduced from \cite[Lemma~1.8]{Marker} that if $I$ is a prime differential ideal of $(\S(W_+), \del_1)$ then there is $f \in I$ such that $I$ is the minimal prime differential ideal containing $f$.    We thank Omar Leon Sanchez for calling our attention to this result.

Overall, this shows that 
differential algebra can be helpful in the study of Poisson ideals of $\S(W_+)$.

\end{remark}

\section{Growth of (Poisson) algebras}\label{DIMENSION}
In this section we first define the {\em Poisson GK-dimension} of a Poisson algebra, and then show that this can be used to compute the GK-dimension of an almost-commutative filtered ring under appropriate conditions.
Finally, we give applications of our general results to $\U(W_+)$.

\subsection{Poisson GK-dimension}
In this subsection we define and give general results on Poisson GK-dimension.  The techniques here are standard, but since the terminology is new we give the proofs in a fairly high level of detail.

We begin with definitions. We work over the fixed ground field  $\kk$, and write $\dim V$ for $\dim_\kk V$ if $V$ is a $\kk$-vector space.
  We first recall some standard definitions from~\cite[Chapter~1]{KL}.  

\begin{defn} \label{def:growth}
Let $f,g$ be monotone increasing functions from $\NN$ to $\RR^+$.  We say $f \leq^* g$ if there are  $c,m \in \NN$  so that $f(n) \leq cg(mn)$ for all but finitely many  $n \in\NN$, and $f \sim g$ if $f \leq^*g$ and $g \leq^*f$.  
We let $\sG(f)$ be the $\sim$-equivalence class of $f$, and write $\sG(f) \leq \sG(g)$ if $f \leq^* g$.

For $0 \leq \gamma  \in \RR$ and $n \in \NN$, let $p_\gamma(n) := n^\gamma$.  Let $\mc{P}(\gamma) := \sG(p_\gamma)$.

If $R$ is a $\kk$-algebra which is generated by a finite-dimensional subspace $V$, define 
\[ \dd_V(n) := \dim \left(\sum_{k=0}^n V^k\right).\]
By \cite[Lemma~1.1]{KL}, the growth $\sG(\dd_V)$ does not depend on the choice of the generating subspace $V$, and we refer to it as {\em the growth of $R$}, written $\sG(R)$.

The {\em Gelfand-Kirillov dimension} or {\em GK-dimension of $R$ } is:
\[ \GK(R) := \inf \{ \gamma: \sG(R) \leq \sP(\gamma)\} = \varlimsup \log_n \dd_V(n),\]
where $V$ is a finite-dimensional  subspace of $R$ which generates $R$ as an algebra.
(The last equality is \cite[Lemma~2.1]{KL}.)

For a not necessarily finitely generated algebra $R$, we define $\GK(R) = \sup_{R'} \GK(R')$, where the supremum is taken over all finitely generated subalgebras $R'$ of $R$.
\end{defn}

Our first task is to define the Poisson GK-dimension of a Poisson algebra.  

\begin{defn}\label{def:Pgrowth}
Let $A$ be a Poisson algebra over $\kk$.
Let $V$ be a subspace of $A$.  We inductively define the subspaces $V^{\{n\}}$ as follows:
\begin{itemize}
\item $V^{\{0\}} := \kk$
\item For $n \in \NN$, define $V^{\{n+1\}} := V V^{\{n\}} + \{ V, V^{\{n\}}\}$
\item In particular $V^{\{1\}} := V$.
\end{itemize}
If $V$ is finite-dimensional, we define
\[ \pd_{V}(n)   := \dim \left(\sum_{k=0}^n V^{\{k\}}\right).\]
\end{defn}

We wish to show $\sG(\pd_V)$ does not depend on $V$ as long as $V$ generates $A$ as a Poisson algebra.  
We first show:

\begin{lemma}\label{lem:Ppowers}
Let $A$ be a Poisson algebra and let $V$ be a subspace of $A$.
For all $a,b \in \NN$, we have $$V^{\{a\}} V^{\{b\}} + \{ V^{\{a\}}, V^{\{b\}}\} \subseteq V^{\{a+b\}}.$$
\end{lemma}
\begin{proof}
We prove the lemma by induction on $a$.
By definition, the lemma holds for $a=1$ and for any $b$.
Suppose now that 
the lemma holds for all $a \leq c$ and for any $b$.
Then
\begin{align*}
V^{\{c+1\}} V^{\{b\}} + \{ V^{\{c+1\}}, V^{\{b\}}\} = &
V V^{\{c\}} V^{\{b\}} & + &
\{ V, V^{\{c\}}\} V^{\{b\}} &+ & 
\{ V V^{\{c\}}, V^{\{b\}} \} &+ &
\{\{ V, V^{\{c\}}\} ,V^{\{b\}}\} \\
& (1) && (2)& & (3) && (4)
\end{align*}

The inclusion $(1) \subseteq V^{\{c+b+1\}}$ is immediate by induction.
We have:
\[ (2) \subseteq V^{\{c\}} \{ V, V^{\{b\}}\} + \{ V, V^{\{c\}}V^{\{b\}}\} \subseteq V^{\{c\}} V^{\{b+1\}}+ \{ V, V^{\{c+b\}}\}, \]
\[ (3) \subseteq V \{ V^{\{c\}}, V^{\{b\}}\} + V^{\{c\}}\{V, V^{\{b\}}\}  \subseteq V V^{\{c+b\}} + V^{\{c\}} V^{\{b+1\}} ,\]
\[ (4) \subseteq \{ V, \{ V^{\{c\}}, V^{\{b\}}\}\} + \{ V^{\{c\}}, \{ V^{\{b\}}, V\}\} \subseteq \{V, V^{\{c+b\}}\} + \{ V^{\{c\}}, V^{\{b+1\}}\}.\]
All of these are contained in $V^{\{c+b+1\}}$ by induction.
\end{proof}

\begin{proposition}\label{prop:PGK1}
For any finite-dimensional spaces $V,W$ which generate $A$ as a Poisson algebra, we have $$\sG(\pd_V) = \sG(\pd_W).$$
\end{proposition}
\begin{proof}
Since $V$ and $W$ generate $A$ as a Poisson algebra, there are positive integers $s,t$ so that
\[ V \subseteq \sum_{j=0}^s W^{\{j\}}, \quad W \subseteq \sum_{k=0}^t V^{\{k\}}.\]
By Lemma~\ref{lem:Ppowers} and induction, for all $n$ we have $V^{\{n\}} \subseteq \sum_{j=0}^{sn} W^{\{j\}}$, and so $\pd_V(n) \leq \pd_W(sn)$.
Likewise, $\pd_W(n) \leq \pd_V(tn)$ and so $\sG(\pd_V)= \sG( \pd_W)$.
\end{proof}

\begin{defn}\label{def:PGK}
If $A$ is a generated as a  Poisson algebra by some finite-dimensional subspace $V$, we define the {\em Poisson GK-dimension of $A$} to be 
\[ \PGK(A) =  \inf \{ \gamma: \sG(\pd_V) \leq \sP(\gamma)\} = \varlimsup \log_n \pd_V(n).\]
By Proposition~\ref{prop:PGK1}, this does not depend on the generating space $V$ chosen.

For an arbitrary  Poisson algebra we define 
$$\PGK(A):= \sup_{A'} \PGK(A'),$$ where the supremum is taken over all finitely generated Poisson subalgebras $A'$ of $A$.
\end{defn}

For Poisson algebras with a sufficiently nice filtration, we can compute Poisson GK-dimension from the growth of the filtration.
Let $A$ be an algebra, and let $\kk = A(0) \subseteq A(1) \subseteq \cdots$ be a filtration of $A$;
recall this means that the $A(i)$ are subspaces of $A$ so that $A(i)A(j) \subseteq A(i+j)$ for all $i,j \in \NN$.
We say the  filtration is {\em exhaustive} if $A = \bigcup A(n)$ and {\em finite} if $\dim A(n) < \infty$ for all $n$.
The filtration is {\em discrete} since $A(k) = (0) $ for all $k < 0$.

\begin{lemma}\label{lem:PGKfilter}
Let $A$ be a finitely generated Poisson algebra, discretely, finitely, and exhaustively filtered by $\kk = A(0) \subseteq A(1) \subseteq \cdots$.
Assume also that $\{A(n), A(m)\} \subseteq A(n+m)$ for all $n, m \in \NN$.
Then $$\PGK A \leq \varlimsup \log_n \dim A(n).$$

If there is $k$ so that 
\beq
\label{goodgrowth} A(n) \subseteq A(k)^{\{n\}} \quad \mbox{ for all $n$},
\eeq
 then $\sG(\dim A(n)) = \sG(\pd_{A(k)})$.  In particular, $\PGK A = \varlimsup \log_n \dim A(n)$.
\end{lemma}
\begin{proof}
Let $V$ be a finite-dimensional subspace of $A$. 
For some $p$ we have $V \subseteq A(p)$, and it follows that $V^{\{n\}} \subseteq A(pn)$ for all $n$.
Thus $\sG(\pd_V) \leq \sG(\dim A(n))$ and the first inequality follows.

Now suppose that 
\eqref{goodgrowth} holds for $k$.
Clearly $A(k)$ generates $A$ as a Poisson algebra.
By the first paragraph and \eqref{goodgrowth},  $\sG(\dim A(n)) = \sG(\pd_{A(k)})$. The final statement follows.
\end{proof}

If the filtration $A(n)$ on $A$ satisfies \eqref{goodgrowth} for some $k$, we say that $A$ has {\em good growth} with respect to the filtration.  

\begin{remark}\label{rem:fg}
If $A$ is finitely generated as an algebra, then  $A(n) \subseteq A(k)^n$ for all $n$ (for some $k$)
 \cite[Lemma 6.1]{KL}. However,  \eqref{goodgrowth} does not seem to follow  from $A$ being finitely generated as a Poisson algebra without extra conditions; see Proposition~\ref{prop:GKW1} and Remark~\ref{rem:goodgrowth}.
\end{remark}
 
To end the subsection, we note that if $A$ is a finitely generated Poisson algebra, then the Poisson GK-dimension of $A$ is also the GK-dimension of $A$ as a module over a certain ring of differential operators.  
We refer to \cite[Chapter~5]{KL} for definitions, see also Proposition~\ref{prop:PGK2}.

If $v \in A$, define $\del_v := \{v, -\}$.  This is a derivation of $A$.

\begin{proposition}\label{prop:PGK2}
Let $A$ be generated as a Poisson algebra by a finite-dimensional subspace $V$.
Let $$D := A \ang{\del_v\ : \ v \in V},$$ considered as a subalgebra of the ring $D(A)$ of differential operators on $A$.
Note that $A$ has a natural left  $D$-module structure. Then $\PGK A = \GK_D A$.
\end{proposition}
\begin{proof}
We write the action of $D$ on $A$ as $D \cdot A$.
Inside $D$, our convention is that $\del_v x = x \del_v + \del_v(x)$ for all $v \in V$, $x \in A$.
Let $\del_V  = \{ \del_v: v \in V\}$.

Without loss of generality, we may assume that $1 \in V$.  
Let $W = V +  \del_V \subseteq D$.  
We claim that $W^n  \supseteq V^{\{n\}}$ for all $n$.  To see this assume that it holds for $n$.
Then 
\[ W^{n+1} \supseteq V W^n + \del_V W^n + W^n \del_V \supseteq V V^{\{n\}} + \{ \del_v (x)  \ : \ v \in V, x \in V^{\{n\}} \} = V^{\{n+1\}}.\]
Since $A = \bigcup V^{\{n\}}$ and $D = A \ang{\del_V}$ we have that $W$ generates $D$ as a $\kk$-algebra.

Note that for any $X \subseteq A$ we have $W \cdot X = VX + \{ V, X\}$, and so $W^n \cdot \kk = V^{\{n\}}$.
Thus 
\[ \GK_D A =\varlimsup \log_n \dim W^n \cdot \kk =  \varlimsup \log_n \dim V^{\{n\}} = \PGK A,\]
by \cite[page 51]{KL}.
\end{proof}

\subsection{Relating GK-dimension and Poisson GK-dimension}

Let $R$ be a finitely and discretely filtered ring so that the associated graded ring $\gr R$ is finitely generated.  
It is standard that $\GK R = \GK (\gr R)$; see \cite[Proposition~6.6]{KL}.  We wish to use a similar technique to understand the GK-dimension of quotients of $\U(W_+)$.  Unfortunately, $\gr \U(W_+) = \S(W_+)$ is not finitely generated as an algebra; however, it is finitely generated as a Poisson algebra, and we will show that we can relate the GK-dimension of (a quotient of) $\U(W_+)$ and the Poisson GK-dimension of the associated graded ring.  

\begin{definition} 
Let $R$ be a finitely generated $\kk$-algebra together with a filtration 
\beq \label{filter}
(0)\subseteq R(0)\subseteq R(1)\subseteq R(2)\subseteq\cdots
\eeq
  that is discrete, finite, and exhaustive.  (Recall that these terms were defined before Lemma~\ref{lem:PGKfilter}.) 
Then $R$ is {\em almost commutative} with respect to this filtration if  $[r_i, r_j]\in R(i+j-1)$ for all $i, j\ge 0$ and $r_i\in R(i), r_j\in R(j)$.\end{definition}

In this subsection we consider an algebra $R$ that is almost commutative with respect to a discrete, finite, exhaustive filtration as in \eqref{filter}.
Let $$A:= \gr R = \bigoplus_{n \geq 0} R(n)/R(n-1).$$
It is standard that $A$ is a graded ring with $A_n:= R(n)/R(n-1)$.
If $r \in R(n) \ssm R(n-1)$, we write $$\gr(r): = r + R(n-1) \in A_n.$$
Since $R$ is almost commutative,  $A$ is commutative and carries a well-defined Poisson bracket: if  $\gr(r) \in A_m$, $\gr(s) \in A_m$, then
\beq \label{bracket} 
 \{ \gr(r), \gr(s)\}  = \begin{cases} 0 & \mbox{if $[r,s] \in R(n+m-2)$} \\
 \gr([r,s]) & \mbox{else}.
 \end{cases}
\eeq
 Thus $$\{ A_m, A_n \} \subseteq A_{n+m-1}.$$ 
Let $A(n) = \gr R(n) = \bigoplus_{k=0}^n A_k$.
Since the filtration is discrete, $\dim A(n) =\dim R(n)$.

We have
\begin{lemma}\label{lem:trivbrack} Let $R$ be an algebra that is almost commutative with respect to a discrete, finite, exhaustive filtration as above.  For any subsets $X,Y \subseteq R$ we have $$\gr[X, Y] \supseteq \{ \gr X, \gr Y\}.$$
\end{lemma}
\begin{proof}
We can reduce the statement to the case $\dim X = \dim Y  = 1$, which is given by \eqref{bracket}.
\end{proof}

Our main result on Poisson GK-dimension is the following:
\begin{proposition}\label{prop:GKandPGK}
Let $R$ be an algebra that is almost commutative with respect to the discrete, finite, exhaustive filtration \eqref{filter}, and let $A = \gr R$ with $A(n) = \gr R(n)$. Then $$\GK(R) \geq \PGK(A).$$
If $A$ has good growth with respect to the filtration $\{A(n)\}$, that is if \eqref{goodgrowth} holds for some $k$, 
then $$\GK(R) = \PGK(A) = \varlimsup \log_n \dim A(n) $$
\end{proposition}

\begin{proof}
Let $V$ be a finite-dimensional subspace of $A$, with $1 \in V$.
Choose a finite-dimensional subspace $W$ of $R$, with $1 \in W$, so that $\gr W \supseteq V$.

We claim that $V^{\{n\}} \subseteq \gr W^n$ for all $n$.  
The claim is true for $n=1$; assume that it holds for $n$.  
Then
\begin{align*}
 V^{\{n+1\}}  & = V V^{\{n\}} + \{ V, V^{\{n\}} \} \\
  & \subseteq (\gr W)(\gr W^n) + \{ \gr W, \gr W^n \} \quad \mbox{ by induction} \\
   & \subseteq \gr W^{n+1} + \gr [W, W^n] \quad \mbox{ by Lemma~\ref{lem:trivbrack}.}
   \end{align*}
  Since $[W, W^n] \subseteq W^{n+1}$, the claim is proved.
  
Now $\dim V^{\{n\}} \leq \dim \gr W^n = \dim W^n$ (since $R$ is discretely filtered). 
Then we have $\sG(\pd_V) \leq \sG(\dd_W)$, so

\[ \varlimsup \log_n \pd_V(n) \leq \varlimsup \log_n \dd_W(n).\]
Taking the supremum over all $V$ and $W$, we obtain that $\PGK A \leq \GK R$.
  
  Assume now that \eqref{goodgrowth} holds for $k$, and let $V= A(k)$ and $W = R(k)$.  
    We claim that $W$ generates $R$ as an algebra; in fact, we claim that $R(n) \subseteq W^n$ for all $n$.
    This is clearly true for $n \leq k$.
       Let $r \in R(n) \ssm R(n-1)$.
    We have $\gr r \in V^{\{n\}} \subseteq \gr W^n$ and so there is $w \in W^n \cap R(n)$ with $r-w \in R(n-1)$.
    By induction, $r - w \in  W^{n-1} \subseteq W^n$ so $r \in W^n $.
    
    Since $W^n \subseteq R(nk)$ we have  
    \[ \varlimsup \log_n \dim R(n) =\varlimsup \log_n \dd_W = \GK R .\]
  But by 
   Lemma~\ref{lem:PGKfilter}, 
  \[ \varlimsup \log_n \dim R(n) = \varlimsup \log_n \dim A(n) = \PGK A,\]
  completing the proof.
    \end{proof}

  \subsection{Consequences for quotients of $\U(W_+)$ and $\S(W_+)$}
  
  We now apply our previous results to quotients of $\U(W_+)$ and $\S(W_+)$.
  First consider a quotient of $\S(W_+)$ by a radical Poisson ideal.
  
 \begin{theorem}\label{thm:GKW3}
 Let $K$ be a nonzero radical Poisson ideal of $\S(W_+)$ and let $A = \S(W_+)/K$.
 Then $$\PGK A = \GK A < \infty.$$
 \end{theorem}
To prove Theorem~\ref{thm:GKW3} we need the following result.
\begin{proposition}\label{prop:Smithcor}
Let $L$ be a finitely generated field extension of $\kk$ of transcendence degree $n$,
and let $D(L)$ be the ring of $\kk$-linear differential operators on $L$.  Then $$\GK_{D(L)}(L) = n.$$
\end{proposition}
\begin{proof}
This is a direct consequence of the methods of \cite{Smith}, although this result does not seem to appear in the literature.

For any subalgebra $A$ of $L$, let $\Delta_A$ be the module of derivations of $A$.
If $\Spec A$ is smooth and affine, then by \cite[Proposition~2.2]{Smith}, $\Delta_A$ is projective and $D(A) $ is the subalgebra of $D(L)$ generated by $A$ and $\Delta_A$:  that is $D(A) =  A[\Delta_A]$.  
It follows that, $D(L) = L[\Delta_L]$, and that $$\GK_{D(L)}(L)= \sup_A \GK_{A[\Delta_A]} A,$$ where the supremum is taken over all finitely generated subalgebras $A$ of $L$ with $Q(A) = L$.  
Since $\operatorname{char} \kk = 0$, by generic smoothness we may enlarge $A$ to obtain a finitely generated algebra  $A' \subseteq L$ with   $\Spec A'$  smooth and $Q(A') = L$.
As $ \Delta_{A'}$ is projective, there is a finitely generated algebra  $A'' $ with $$A' \subseteq A'' \subseteq L$$ so that $\Delta_{A''}$ is free over $A''$, and it suffices to prove that $\GK_{A''[\Delta_{A''}]} (A'') = n$.  Let $D = A''[\Delta_{A''}]$.

Let $c_1, \dots, c_n$ be an $A''$-basis for $\Delta_{A''}$ and let $C $ be the $\kk$-span of $(c_1, \dots, c_n)$.
As a left $D$-module  we have $$A'' \cong D/ D C.$$
As in \cite[Section 4]{Smith} we may choose a finite-dimensional generating subspace $V$ of $A''$ so that $$[C,C] \subseteq VC~{\rm~and~}~C(V) \subseteq V^2.$$  
Let $W = V \oplus C$, which generates $D$.
By \cite[Theorem~4.4]{Smith}, for all $k$ we have $W^k V = V^k$ as subspaces of $A''$.
Thus $\GK_D A''  = \GK A'' = n$.
\end{proof}
\begin{proof}[Proof of Theorem~\ref{thm:GKW3}]
 Let $y_i$ be the image of $x_i = \gr e_i$ in $A$.  
 As a Poisson algebra, $A$ is generated by $y_1$ and $y_2$.
 For $i = 1,2$, let $$\del_i = \{ y_i, -\} \in \Der(A),$$ and let $D = A\ang{\del_1, \del_2} \subseteq D(A)$.
 We have $$\PGK A = \GK_D A$$ by Proposition~\ref{prop:PGK2}.
 
 Clearly $\PGK A \geq \GK A$, so it suffices to prove that $\GK_D A \leq \Kdim A,$
 which is $\GK A$ by Corollary~\ref{cor2}. 
 
 We first assume that $K$ is prime.
 By Lemma~\ref{Lns}(b), there is some nonzerodivisor $p \in A$ so that $$A \hra A' = A[p^{-1}]$$ and $A'$ is a finitely generated algebra.
 As $A'$ is also Poisson,  $D$ also acts on $A'$.
Let $$L:= Q(A) = Q(A'),\hspace{10pt}n:=\trdeg L.$$  
We have $n = \Kdim A' = \Kdim A$.  
Then we have $$\GK_D A \leq \GK_D L \leq \GK_{D(L)} L  = n = \Kdim A,$$ by Proposition~\ref{prop:Smithcor}.

For general $K$, by Corollary~\ref{Cns} we have $K = \mf p_1 \cap \dots \cap \mf p_m$, where the $\mf p_i$ are prime Poisson ideals and are therefore $D$-stable.  Thus $D$ acts on $A_i = A/\mf p_i$; let the  operators $\overline{\del_{1,2}}$ on $A_i$ be  induced from the action of $\del_{1,2}$ on $A$.

 Let $D_i = A_i \ang{\overline \del_1, \overline \del_2}$.  Clearly $\GK_D A_i = \GK_{D_i} A_i$, which is $\Kdim A_i$ by the prime case.
 Thus, applying \cite[Proposition~5.1(a)]{KL}, 
 \[ 
\GK_D A \leq \GK_D \left( \bigoplus_i A_i \right)= \max_i \GK_D A_i = \max_i \Kdim A_i.\]
  But this is $\Kdim A$ by definition.
\end{proof}

We note that the conclusion of Theorem~\ref{thm:GKW3} can fail for non-radical ideals.  
Indeed, let $$I = (x_i x_j:  i, j \in \ZZ_{\geq 1}){\rm~and~}~A = \S(W_+)/I.$$  
It is easy to see that $I$ is a Poisson ideal, and that $\GK A = 0$ and $\PGK A = 1$.

  We now derive results for  quotients of $\U(W_+)$.
  Recall that $\U(W_+)$ is both graded by degree and filtered by order of operators (as the enveloping algebra of a Lie algebra), and we write the degree and order of an element $f$ respectively as $d(f)$ and $o(f)$. 
  Thus $d(e_n) = n$ and $o(e_n) =1$.
  The symmetric algebra $$\S(W_+) = \gr_o \U(W_+)$$ is graded both by $d$ and by $o$, or alternatively is $\NN\times \NN$-graded.
 Note that if we write $do(f) = d(f) + o(f)$ for the {\em degree-order filtration} on $\U(W_+)$, that $\gr_{do} \U(W_+)  = \S(W_+)$ as well, although of course the induced grading is different.

 Our first result is that $d$-graded ideals of $\U(W_+)$ automatically give rise to Poisson ideals of $\S(W_+)$ so that the quotients  have good growth in the sense of \eqref{goodgrowth}.
 
 \begin{proposition}\label{prop:GKW1}
 Let $I$ be a $d$-graded ideal of $\U(W_+)$ and let $R = \U(W_+)/I$.
 \begin{itemize}
 \item[(a)] The $do$-filtration on $\U(W_+)$ induces a discrete, finite, exhaustive filtration on $R$ with respect to which $R$ is almost commutative.
 
 \item[(b)] Let $A = \gr_{do}(R)$.
 Then $$A  \cong \S(W_+)/\gr_{do}(I) \cong \S(W_+)/\gr_o(I).$$ 
 
 \item[(c)] $A$ has good growth with respect to the filtration induced from the $do$-filtration on $R$.
 \end{itemize}
 \end{proposition}
 \begin{proof}
 Let $U=\U(W_+)$, let $$U(n) = \{ f \in U \ : \ do(f) \leq n\},$$ and let $I(n) = I \cap U(n).$
 Let $$R(n) = (U(n)+I)/I \cong U(n)/I(n).$$ 
 It is immediate that the $R(n) $ give a discrete, finite, exhaustive filtration on $R$, which we will  refer to as  the $do$-filtration on $R$.  
 Since $\U(W_+)$ is almost commutative with respect to the $do$-filtration, clearly $R$ is almost commutative with respect to the $do$-filtration on $R$, and thus (a) holds.

 As $A = \gr_{do}(R)$, then 
$$ A_n = 
 \frac{(U(n)+I)/I}{(U(n-1)+I)/I} \cong \frac{U(n)}{U(n-1)+I(n)}.$$
 Thus $\gr_{do}(I)$ is the kernel of the natural surjection from $$\S(W_+) = \gr_{do}(U) \to A.$$
 As $I$ is $d$-graded we have $\gr_{do}(I) = \gr_o(I)$ as ideals of $\S(W_+)$.  This proves (b).  
 
 For (c), it suffices to show that $\U(W_+)$ has good growth.  
 Let $$V = \kk\cdot(x_1,x_2) \subseteq A(3).$$
 Let $y \in A(n) \ssm A(n-1)$.  We must show that $y \in V^{\{n\}}$. 
 
Notice that $d(y) \leq do(y) \leq n$.
 We can write $y$ as a sum of monomials of the form
 $ e_{i_1} e_{i_2} \cdots e_{i_\ell}$,
 where $i_1 \leq i_2 \leq \cdots \leq i_\ell$ and $\sum i_j \leq n$.
To show that $y \in V^{\{n\}}$, it suffices to show that $e_m \in V^{\{m\}}$ for all $m$.
 This is true for $m =1, 2$; and for $m \geq 3$ we have for some $\lambda \in \kk \ssm 0$ that $e_m = \lambda \{ e_1, e_{m-1}\} \in \{ V, V^{\{m-1\}}\}$ by induction.
 \end{proof}
 
 \begin{remark}\label{rem:altfilter}
  There is an alternate filtration on $A$ defined via  the $d$-grading:  let $$\mc{F}^i\S(W_+) = \{ f \in \S(W_+): d(f) \leq i\}$$ and define a filtration $\mc{F}^iA$ on $A$ accordingly.  Then the argument above shows that $$\mc{F}^nA \subseteq V^{\{n\}} \subseteq (\mc{F}^2 A)^{\{n\}},$$ so $A$ also has good growth with respect to the filtration 
 \[ (0) \subseteq \mc{F}^0 A \subseteq \mc{F}^1A \subseteq \dots.\]
 \end{remark}

 Combining  the previous proposition with earlier results, we obtain:
 \begin{theorem}\label{thm:GKW2}
 Let $J$ be an ideal of $\U(W_+)$.  Let $I = \gr_d(J) \triangleleft \U(W_+)$
 and let $$K = \gr_o(I) = \gr_{do}(I) \triangleleft \S(W_+).$$
 Then $\GK \U(W_+)/J = \GK \U(W_+)/I  = \PGK \S(W_+)/K$.
 \end{theorem}
 
\begin{proof}
 That $\GK \U(W+)/J = \GK \U(W_+)/I $ is \cite[Proposition~6.6]{KL}, since the $d$-grading on $\U(W_+)$ induces a discrete finite exhaustive filtration on $\U(W_+)$.
 By Propositions~\ref{prop:GKandPGK} and \ref{prop:GKW1}, $$\GK \U(W_+)/I = \PGK \S(W_+)/K.$$ \end{proof}

\begin{remark}\label{rem:goodgrowth}
If $R = \bigoplus_{j \in \NN} R_j$ is any $\NN$-graded ring that also has  a discrete finite exhaustive filtration 
\[
(0)\subseteq R(0)\subseteq R(1)\subseteq R(2)\subseteq\cdots
\]
 with respect to which $R$ is almost commutative, and so that each $R(n) = \bigoplus_j (R(n) \cap R_j)$ is a graded vector space, then the argument above shows (by adding the two gradings on $A := \gr R$) that if $A $ is finitely generated as a Poisson algebra then $A$ has good growth and therefore that $\GK R = \PGK A$.
 \end{remark}

Finally, we have:
\begin{corollary}\label{cor4}
Let $J$ be a nontrivial ideal of $\U(W_+)$ so that $\gr_o(\gr_d(J))$ is radical.
Then \[\GK \U(W_+)/J <\infty.\]
\end{corollary}
\begin{proof}
This follows directly from Theorems~\ref{thm:GKW2} and \ref{thm:GKW3}.
\end{proof}

We conjecture  that the conclusion of Corollary~\ref{cor4} holds for any  nontrivial ideal of $\U(W_+)$; see  Conjecture~\ref{conj1}.  
Note that by Theorem~\ref{thm:GKW2}, it suffices to prove that $\PGK \S(W_+)/K < \infty$ for any $d$-graded and $o$-graded Poisson ideal $K$ of $\S(W_+)$.

\section{Quotients by quadratic elements}\label{QUADRATIC}

The results in the previous sections may be thought of as providing evidence that Conjecture~\ref{conj1} holds and thus that nontrivial ideals of $\U(W_+)$ and Poisson ideals of $\S(W_+)$ are large. 
 If this is the case, it is natural to expect  that $\U(W_+)$ satisfies the ascending chain condition on ideals: in other words,  that Conjecture~\ref{conj2} holds.
 (Examples such as \cite[Theorem~2.14]{Bell} show that finite GK-dimension does not even imply the ascending chain condition on prime ideals, so we phrase this as an expectation, not a formal consequence.)
 
 In this section we study Conjecture~\ref{conj1} and Conjecture~\ref{conj2} for ideals containing elements of order two.
 We first prove that $\S^2(W_+)$ is a noetherian representation of $W_+$, from which it follows trivially that  $\S(W_+)$ satisfies the ascending chain  condition  on Poisson ideals generated by elements of order two. 
  (As a byproduct, we show that $\S^2(W_+)$ is GK-2 critical.)
 As a consequence of our methods, we show that any quotient of $\U(W_+)$ by an ideal containing a nontrivial element of order one or two  has finite GK-dimension.  
 
\subsection{Noetherianity of  $\S^2(W_+)$}
Before proving that $\S^2(W_+)$ is noetherian, we 
show that the adjoint representation of $W_+$ is noetherian. This is implied by the following lemma.
\begin{lemma}\label{Lalev} Let $\mathfrak l$ be a nonzero submodule of $W_+$. Then
 $e_n\in\mathfrak l$ for some $n$.  As a result, 
$\dim (W_+/\mathfrak l)<\infty$.
\end{lemma}
We would thank Jacques Alev for the proof of this result.
\begin{proof}
 Fix $x\in\mathfrak l\backslash 0$. Then there are $$0<i_1<i_2<\dots<i_s\in\ZZ_{\geq 1}{\rm~and~}a_1,\dots, a_s\in\kk \backslash 0$$
such that
$$x=a_1e_{i_1}+\cdots +a_se_{i_s}.$$
We say that $s$ is the {\em length} of $x$.
If $s > 1$ then $[e_{i_1}, x]$ is nonzero and has length $< s$.  By induction, there is some $e_n$ in the Lie ideal generated by $x$.
It is an easy computation that the Lie ideal generated by  $e_n$ contains $e_{\geq n+2} = \{ e_j : j \geq n+2 \}$.
\end{proof}
By the above, any nontrivial ideal of $W_+$ 
has cofinite dimension, and thus   $W_+$ is noetherian as a Lie algebra and  as a $ W_+$-module.
The main result of this subsection is:

\begin{theorem}\label{thm:S2}  The $W_+$-module $\S^2(W_+)$ is  noetherian.
\end{theorem}

\begin{proof}
Our strategy is to put a monomial order on $\S^2(W_+)$ and then for a submodule $M \leq \S^2(W_+)$, describe the combinatorial structure of the set of leading terms of elements of $M$.

We first establish notation.  
A basis for $\S^2(W_+)$ is $\{x_i x_j : 1 \leq i \leq j\}$.
Let $ \Gamma = \{ (i,j) \in \ZZ^2 : 1 \leq i \leq j\}$, so $\Gamma$ is a grading  semigroup for $\S^2(W_+)$ as a vector space. 
Define an order $\prec$ on $\Gamma$ by setting $(i,j) \prec (k, \ell)$ if and only if either $i +j< k+\ell$ or $i+j=k+\ell$ and $j < \ell$.
Note that $\prec$ is a well-ordering, and that
 the smallest elements of $\Gamma$ are 
  \[ (1,1) \prec (1,2)  \prec (2,2)  \prec  (1,3)  \prec  (2,3)  \prec  (1,4)  \prec  (3,3) \prec \dots .\]
  If $f \in \S^2(W_+)$, let $\gamma(f)$ be the degree of the leading term of $f$ with respect to the order on $\Gamma$; so 
  $$\gamma (2 x_1x_5 + x_3^2) = (1,5).$$

  For $n \in \NN$, let $$\U(W_+)_n := \{ f \in \U(W_+) \st \mbox{ $f$ is $d$-homogeneous of degree $n$ } \}.$$
  Our  convention going forward is that if $X$ is a $d$-graded object, then $X_d = \{ x\in X : d(x) = d\}$.
  
  We then have:

 \begin{lemma}\label{lem:S2one}
  For $1 \leq i \leq j \in \NN$, let $\VV_{ij} = \spn(x_i x_{j+6}, x_{i+1} x_{j+5}, x_{i+2}x_{j+4}, x_{i+3} x_{j+3})$ and let 
  \[\pi_{ij}:  \S^2(W_+) \to \VV_{ij}\]
   be the projection.
 There is some integer $N > 1$ so that for all $N \leq i \leq j$ and for all $d$-homogeneous $f \in \S^2(W_+)$ with $\gamma(f) = (i,j)$, the linear map
 \[ \pi_f:  \U(W_+)_6 \to \VV_{ij},  \quad p \mapsto \pi_{ij}(p \cdot f)\]
 is surjective.
   \end{lemma}
  
  Assume Lemma~\ref{lem:S2one} for the moment.   For any $n \in \NN$, let 
  \[S(n ) = \kk \cdot (x_i x_j: n \leq i \leq j)\]
   and let 
 $S'(n) = \kk \cdot (x_i x_j:  
\mbox{ $n  \leq i \leq j $ with $n \neq j$ }).$
There is  a chain of $W_+$-modules
 \[ \S^2(W_+) = S(1) \supset S'(1) \supset S(2) \supset S'(2) \supset \dots .\]
 Now $S(n)/S'(n)$ is $1$-dimensional, and $S'(n) /S(n+1)$ is spanned by $$\{ f_m = x_n x_m + S(n+1): m > n\}.$$
 Since $e_k \cdot f_m = (m-k) f_{m+k}$, thus $S'(n) /S(n+1)$ is isomorphic to a subrepresentation of $W_+$ and is  noetherian.
 
 Let $S = S(N)$, where $N$ is 
  the constant given in Lemma~\ref{lem:S2one}.
  By the above, $\S^2(W_+)/S$ is noetherian, so 
   it suffices to prove that $S$ is noetherian. 
   Since $S$ is  $\NN$-graded by degree, by \cite[Proposition~1.6.7]{MR} it suffices to prove that any $d$-graded submodule is finitely generated.
    
     Let $M$ be a $d$-graded submodule of $S$, and consider $\gamma(M) \subseteq \Gamma$. 
It follows from Lemma~\ref{lem:S2one} that if $f \in M$ with $\gamma(f) = (i,j)$, then there is $p \in \U(W_+)_6$ so that $\gamma(p \cdot f) = (i+3, j+3)$.  Further, since $i \geq 2$, we see that $\gamma(e_1 \cdot f) = (i,j+1)$.
It follows that if    $\Sigma $ is the sub-semigroup of $\NN \times \NN$ generated by $\{(0,1),(3,3)\}$, then $\gamma(M)$ is a $\Sigma$-subrepresentation of $\Gamma$.

Since  $\Sigma$ is finitely generated and abelian and $\Gamma$ is generated over $\Sigma$ by $\{(1,1), (2,2), (3,3)\}$, thus $\Gamma$ is a noetherian representation of $\Sigma$.  
 Thus there are homogeneous $f_1, \dots, f_k \in M$ so that $\gamma(M)$ is generated as a $\Sigma$-module by $\gamma(f_1) , \dots, \gamma(f_k)$.  
 Let $M'$ be the $W_+$-subrepresentation of $M$ generated by $f_1, \dots, f_k$.  We claim that $M' = M$.
 
 To see this, note that $\gamma(M')\subseteq \gamma(M)$ is a $\Sigma$-module containing $\gamma(f_1),\dots, \gamma(f_k)$; thus $\gamma(M') = \gamma(M)$.  Suppose that there exists 
 homogeneous $f \in M \ssm M'$; we may assume that  $\gamma(f) $ is  minimal in the $\prec$ order among all such $f$.
 By the above, there is $f' \in M'$ with $\gamma(f') = \gamma(f)$.  
Then $f - f' \in M \ssm M'$, and  $\gamma(f-f') \prec \gamma(f)$, contradicting our choice of~$f$.
 \end{proof}

  It remains to prove Lemma~\ref{lem:S2one}.  
  
  \begin{proof}[Proof of Lemma~\ref{lem:S2one}]
 The proof is computational.   We write  $ e_{\blambda} = e_{\lambda_1} \dots e_{\lambda_k} $ where $\blambda = (\lambda_1 \leq \lambda_2 \leq \cdots \leq \lambda_k)$ is a partition. 
 (In this proof, $\blambda$ will be a partition of 6, but later we will use this notation for a general partition.)
  Thus, for example, $e_{114} = e_1^2 e_4$.

   We have:
\begin{align*}
  e_6 \cdot x_i x_j &=(j-6) x_i x_{j+6} + (i-6) x_{i+6} x_j\\
 e_{1 5} \cdot x_i x_j =& (j+4)(j-5) x_i x_{j+6} + (i-1)(j-5)x_{i+1}x_{j+5} + (i-5)(j-1) x_{i+5}x_{j+1} + (i+4)(i-5) x_{i+6}x_j\\
   e_{24} \cdot x_i x_j =& (j+2)(j-4) x_i x_{j+6} + (i-2)(j-4) x_{i+2}x_{j+4} + (i-4)(j-2) x_{i+4}x_{j+2} + (i+2)(i-4) x_{i+6} x_j \\
  e_{114} \cdot x_i x_j =& (j+4)(j+3)(j-4) x_i x_{j+6} + 2(i-1)(j+3)(j-4) x_{i+1} x_{j+5} + i(i-1)(j-4) x_{i+2} x_{j+4}  \\
   & \quad + (i-4)j(j-1) x_{i+4}x_{j+2} + 2(i+3)(i-4)(j-1) x_{i+5}x_{j+1} + (i+4)(i+3)(i-4) x_{i+6} x_j \\
  e_{33} \cdot x_i x_j  =&j (j-3) x_i x_{j+6} + 2 (i-3)(j-3) x_{i+3}x_{j+3} + i(i-3) x_{i+6} x_j\\
  e_{123} \cdot x_i x_j =& (j+4)(j+1)(j-3) x_i x_{j+6} + (i-1) (j+1)(j-3) x_{i+1} x_{j+5} + (i-2)(j+2)(j-3) x_{i+2} x_{j+4} \\
  &  \quad + \left[ (i+1)(i-2)(j-3) + (i-3) (j+1)(j-2) \right] x_{i+3} x_{j+3} + (i+2)(i-3) (j-2) x_{i+4} x_{j+2}  \\
   &  \quad + (i+1)(i-3) (j-1) x_{i+5} x_{j+1} + (i+4) (i+1)(i-3) x_{i+6} x_j 
     \end{align*}
   \begin{align*}
e_{1113} \cdot x_i x_j =& (j+4)(j+3)(j+2)(j-3) x_i x_{j+6} + 3 (i-1) (j+3) (j+2)(j-3) x_{i+1} x_{j+5} \\
    & \quad + 3 i(i-1)(j+2) (j-3) x_{i+2}x_{j+4}  + \left[ (i+1)i(i-1)(j-3) + (i-3)(j+1) j(j-1) \right] x_{i+3} x_{j+3} \\
     & \quad +3(i+2)(i-3) j (j-1) x_{i+4} x_{j+2}  + 3 (i+3)(i+2)(i-3) (j-1) x_{i+5} x_{j+1} \\
      & \quad + (i+4)(i+3)(i+2)(i-3) x_{i+6} x_j \\
  e_{222} \cdot x_i x_j  =&  (j+2) j(j-2) x_i x_{j+6} + 3 (i-2) j(j-2) x_{i+2} x_{j+4} + 3 i(i-2) (j-2) x_{i+4} x_{j+2}\\
  &  \quad + (i+2) i (i-2) x_{i+6} x_j \\
   e_{1122} \cdot x_i x_j =& (j+4)(j+3)j (j-2) x_i x_{j+6} + 2 (i-1)(j+3)j(j-2) x_{i+1} x_{j+5} \\
   & \quad + \left[ i(i-1) j(j-2) + 2 (i-2) (j+2)(j+1)(j-2)\right] x_{i+2} x_{j+4} \\
    & \quad  + 4 (i+1)(i-2) (j+1)(j-2) x_{i+3} x_{j+3} \\
    & \quad+ \left[ 2(i+2)(i+1)(i-2)(j-2) + i(i-2)j(j-1) \right] x_{i+4} x_{j+2} \\
   & \quad  + 2(i+3) i (i-2) (j-1) x_{i+5} x_{j+1}   + (i+4)(i+3)i (i-2) x_{i+6} x_j \\
   e_{11112} \cdot x_i x_j =& (j+4)(j+3)(j+2)(j+1)(j-2) x_i x_{j+6}  + 4 (i-1) (j+3) (j+2)(j+1)(j-2) x_{i+1} x_{j+5} \\
   & \quad + \left[6 i(i-1)(j+2)(j+1)(j-2) + (i-2) (j+2)(j+1)j(j-1) \right] x_{i+2} x_{j+4} \\
  & \quad + \left[4 (i+1) i(i-1) (j+1)(j-2) + 4 (i+1)(i-2) (j+1) j (j-1) \right] x_{i+3} x_{j+3}  \\
   & \quad + \left[ (i+2)(i+1) i(i-1) (j-2) +6 (i+2)(i+1)(i-2) j(j-1) \right] x_{i+4} x_{j+2 } \\
   & \quad + 4(i+3)(i+2)(i+1)(i-2) (j-1) x_{i+5} x_{j+1} + (i+4)(i+3)(i+2)(i+1)(i-2) x_{i+6} x_j \\
   e_{111111}\cdot x_i x_j =& (j+4)(j+3)(j+2)(j+1)j(j-1) x_i x_{j+6} + 6 (i-1) (j+3)(j+2)(j+1) j (j-1) x_{i+1} x_{j+5} \\
    & \quad + 15 i(i-1) (j+2)(j+1) j (j-1) x_{i+2} x_{j+4} + 20 (i+1) i(i-1) (j+1)j (j-1) x_{i+3} x_{j+3} \\
   & \quad + 15 (i+2)(i+1) i (i-1) j (j-1) x_{i+4} x_{j+2} + 6(i+3)(i+2)(i+1) i (i-1)(j-1) x_{i+5} x_{j+1} \\
    & \quad +  (i+4)(i+3)(i+2)(i+1) i (i-1) x_{i+6} x_j.
\end{align*}

 We rewrite these computations by defining  vectors $v_0, \dots, v_6$ in $\ZZ[i,j]^{11}$ so that $v_k$ consists of the coefficients of $x_{i+k} x_{j+6-k}$ in the expressions above; in other words we have the matrix equation  
 \beq\label{FOO} \sum_{\blambda \vdash 6} \alpha_{\blambda} e_{\blambda} \cdot x_ix_j =
\underline{\alpha} \begin{bmatrix} v_0 & v_1 &  v_2 &v_3 &v_4 &v_5 &v_6\end{bmatrix} 
 \begin{bmatrix} x_i x_{j+6} \\ x_{i+1} x_{j+5} \\ x_{i+2} x_{j+4} \\ x_{i+3} x_{j+3} \\ x_{i+4} x_{j+2} \\ x_{i+5} x_{j+1} \\ x_{i+6} x_{j} \end{bmatrix}, \eeq
 defining an element of $\S^2(W_+)$.  
  Explicitly, 
  \[ v_0=\begin{pmatrix} 
  	(j-6)\\	      	   	     	       
	  (j+4)(j-5)\\                   
	  (j+2)(j-4)\\     	    	 
	  (j+4)(j+3)(j-4)\\         	  
	  j(j-3)\\	                 
	  (j+4)(j+1)(j-3)\\             
	  (j+4)(j+3)(j+2)(j-3)\\       
	  (j+2)j(j-2)\\    	      	   
	  (j+4)(j+3)j(j-2)\\	       	   
	  (j+4)(j+3)(j+2)(j+1)(j-2)\\    
	  (j+4)(j+3)(j+2)(j+1)j(j-1)  
	\end{pmatrix}, \quad 
v_1=\begin{pmatrix} 0\\	      	   	     	       
	  (i-1)(j-5)\\                   
	  0\\     	   	   	 
	  2(i-1)(j+3)(j-4)\\         	  
	  0\\	                 
	  (i-1)(j+1)(j-3)\\             
	  3(i-1)(j+3)(j+2)(j-3)\\       
	  0\\    	      	    
	  2(i-1)(j+3)j(j-2)\\	       	   
	  4(i-1)(j+3)(j+2)(j+1)(j-2)\\    
	  6(i-1)(j+3)(j+2)(j+1)j(j-1)  
	 \end{pmatrix}, \mbox { etc.}\]
Note that the $v_k$ depend on $i$ and $j$.

Let $q_{ab}$ be the coefficient of $x_a x_b$ in \eqref{FOO}.  
If $j-i>6$ then the elements $x_i x_{j+6}, \dots, x_{i+6} x_j$ are distinct, and $q_{ab}$ may be read directly from \eqref{FOO}.  Slightly more generally, in fact, 
\beq\label{BAR}
\mbox{if $j > i+k$, then } q_{i+k, j+6-k} = \underline{\alpha} v_k .
\eeq 
However, if $j-i$ is small, \eqref{BAR} needs to be modified; for example, if $i=j$ then $q_{i,i+6} = \underline{\alpha} (v_0+v_6)$.

Let $f$ satisfy the hypotheses of the Lemma and write
\[ f = e_i e_j + \sum_{k=1}^{\lfloor( j - i)/2 \rfloor} \beta_k e_{i+k} e_{j-k}.\]
First assume that $j > i+6$.  
It follows from \eqref{BAR} that for $p = \sum_{\blambda \vdash 6} \alpha_{\blambda} e_{\blambda}$ we have
\beq
\label{jan} \pi_f(p) = \underline{\alpha} BCX\eeq
where 
\[ C = \begin{bmatrix} 1 & 0 & 0 & 0 \\ 0 & 1 & 0 & 0 \\ 0 & 0 & 1 & 0 \\ 0 & 0&0&1 \\
 0 & \beta_1 & 0 & 0 \\ 0&0&\beta_1 & 0 \\ 0 & 0 & 0 & \beta_1 \\
 0& 0 & \beta_2 & 0 \\ 0 & 0 & 0& \beta_2 \\
 0 & 0 & 0 & \beta_3
 \end{bmatrix}, \quad X = \begin{pmatrix} x_i x_{j+6} \\ x_{i+1} x_{j+5} \\ x_{i+2} x_{j+4} \\ x_{i+3} x_{j+3} \end{pmatrix},\]
and $B$ is the matrix with columns
\beq \label{ten}
 \begin{split} 
 v_{0}(i,j),v_{1}(i,j), v_{2}(i,j), v_{3}(i,j), v_{0}(i+1,j-1), v_{1}(i+1, j-1), \\
 v_{2}(i+1,j-1), v_{0}(i+2,j-2),v_{1}(i+2,j-2), v_{0}(i+3,j-3)
 \end{split}
 \eeq

 The statement of the proposition is equivalent to the statement that $BC$ has rank 4, and for this, since $C$ clearly has rank 4, it is sufficient that $B$ has (full) rank 10; in other words, we claim that for $i \gg 0$, the vectors in \eqref{ten} are linearly independent.   
 
 Let $X$ be the locus in the rational  $(i,j)$-plane $\Spec \Qbar[i,j]$ where the vectors \eqref{ten} are linearly independent.
  If $X \neq \Spec \Qbar[i,j]$, then $\Supp X$ consists of finitely many curves and finitely many isolated points, by primary decomposition.
  Computing in Macaulay2 (see Routine~A.1), we see that $X \neq \Spec \Qbar[i,j]$ and that these finitely many curves are the lines $i=-1,i=0, i=1, j=-1, j=1,$ and $i=j-3$.
  Our assumption that $j > i+6$ means that the condition $i=j-3$ is vacuous.
  Thus for $j-6 > i >1$, we avoid all of these curves, and increasing $i$ further we may avoid the finitely many isolated points in $\Supp X$.
  Thus there is some $N$ so that for $j-6 > i > N$, the vectors \eqref{ten} are linearly independent, and Lemma~\ref{lem:S2one} holds.
  Note that we do not need to compute the 0-dimensional components of $X$ unless we want to calculate $N$ exactly.
  
  This is the general case.  We now suppose that $j-i$ is small.    
  If $j=i+6$ we must modify the final column of $B$, replacing \eqref{jan} by
 \[ 
  \pi_f(p) = \underline{\alpha} B_6  C X,
 \] 
  where $B_6$ is the matrix whose columns are
  $$v_{0}(i,i+6),~v_{1}(i,i+6),~v_{2}(i,i+6),~v_{3}(i,i+6),~v_{0}(i+1,i+5),~v_{1}(i+1, i+5),$$
$$~v_{2}(i+1,i+5),~v_{0}(i+2,i+4),~v_{1}(i+2,i+4),{\rm~and~}(v_{0} + v_{6})(i+3,i+3). $$
By the Macaulay2 computation in Routine~A.2, this holds for $i \gg 0$, using similar arguments to those in the proof of the general case.

  If $j=i+5$ or $j=i+4$ then 
\[ f = x_i x_{j} + \beta_1 x_{i+1} x_{j-1} + \beta_2 x_{i+2} x_{j-2}.\]
If $j=i+5$,  \eqref{jan} is replaced by
\[ \pi_f(p ) = \underline{\alpha} B_5 C' X,\]
where
\[ C' = \begin{bmatrix} 1 & 0 & 0 & 0 \\ 0 & 1 & 0 & 0 \\ 0 & 0 & 1 & 0 \\ 0 & 0&0&1 \\
 0 & \beta_1 & 0 & 0 \\ 0&0&\beta_1 & 0 \\ 0 & 0 & 0 & \beta_1 \\
 0& 0 & \beta_2 & 0 \\ 0 & 0 & 0& \beta_2 
 \end{bmatrix}
 \]
 and $B_5$ has columns
  $$v_{0}(i,i+5),~v_{1}(i,i+5),~v_{2}(i,i+5),~v_{3}(i,i+5),~v_{0}(i+1,i+4),$$
$$v_{1}(i+1, i+4),~v_{2}(i+1,i+4),~v_{0}(i+2,i+3),{\rm~and~}(v_{1} + v_{6})(i+2,i+3).$$ 
  Again, it suffices to prove that $B_5$ has full rank for $i \gg 0$.
  This follows from the computation in Routine~A.3.
  
  If $j= i+4$ then \eqref{jan} becomes
  \[ \pi_f (p ) = \underline{\alpha} B_4 C' X,\]
  where $B_4$ has columns 
  \[ v_0(i, i+4),~v_1(i, i+4),~v_2(i, i+4),~v_3(i, i+4),~v_0(i+1,i+3),\]
  \[v_1(i+1,i+3),~(v_2+v_6)(i+1, i+3),~(v_0+v_6)(i+2,i+2),~(v_1+v_5)(i+2,i+2).\]
  This follows from the computation in Routine~A.4.
  
  If $j=i+3$ then \eqref{jan} becomes
  \[ \pi_f(p ) = \underline{\alpha} B_3 C'' X,\]
  where
\[ C'' = \begin{bmatrix} 1 & 0 & 0 & 0 \\ 0 & 1 & 0 & 0 \\ 0 & 0 & 1 & 0 \\ 0 & 0&0&1 \\
 0 & \beta_1 & 0 & 0 \\ 0&0&\beta_1 & 0 \\ 0 & 0 & 0 & \beta_1 
 \end{bmatrix}
 \]
 and $B_3$ has columns
   \[ v_0(i, i+3),~v_1(i, i+3),~v_2(i, i+3),~(v_3+v_6)(i, i+3),~v_0(i+1,i+2),~(v_1+v_6)(i+1,i+2),~(v_2+v_5)(i+1, i+2).\]
 This follows from Routine~A.5.  
 
 If $j = i+2 $ then we have
 \[ \pi_f(p ) = \underline{\alpha} B_2 C'' X,\]
 where $B_2$ has columns
  \[ v_0(i, i+2),~v_1(i, i+2),~(v_2+v_6)(i, i+2),~(v_3+v_5)(i, i+2),~(v_0+v_6)(i+1,i+1),~(v_1+v_5)(i+1,i+1),~(v_2+v_4)(i+1, i+1).\]
 This follows from Routine~A.6.  
 
  If  $j=i+1$ or $j=i$, then $f = x_i x_j$.  We have 
  \[  \pi_f(p ) = \begin{cases} B_1  X & \quad j =i+1 \\
  B_0 X & \quad j = i,\end{cases} \]
  where 
  \[ B_1 = \begin{bmatrix} v_0(i,i+1) & (v_1+v_6)(i,i+1) & (v_2+v_5)(i,i+1) & (v_3+v_4)(i,i+1) \end{bmatrix},\]
  and
  \[ B_0 = \begin{bmatrix} (v_0+v_6)(i,i) & (v_1 + v_5)(i,i) & (v_2+v_4)(i,i) & v_3(i,i)\end{bmatrix}.\]
  
 The result follows similarly from the computations in Routines A.7 and~A.8.

 \end{proof}

 \begin{corollary}\label{cor:new}
 The algebra $\S(W_+)$ satisfies the ascending chain condition  on Poisson ideals generated by quadratic elements. 
 \qed
 \end{corollary}

 \begin{remark}\label{rem:Sm} 
 Each $\S^m(W_+)$ is also $d$-graded. 
 Fix $m$.  For any $d$, $\dim \S^m(W_+)_d = O(m^{d-1})$.  
 On the other hand $\dim \U(W_+)_d  = P(d)$, the partition number of $d$.
 We do not know if all $\S^m(W_+)$ are finitely generated, though it is certainly plausible, since if $d$ is sufficiently large $\dim \U(W_+)_d$ is much larger than $\dim \S^n(W_+)_d$.    In fact, we  conjecture that $\S^m(W_+)$ is noetherian for all $m$.
 
 Note that if $f \in \S^m(W_+)$ and $\{(f)\}$ is the smallest Poisson ideal containing $f$, then $\{(f)\} \cap \S^m(W_+)$ is equal to the subrepresentation of $\S^m(W_+)$ generated by $f$.
 Thus our conjecture would follow if the ascending chain condition holds for Poisson ideals of $\S(W_+)$; note this last is stronger than Conjecture~\ref{conj2}. 
 \end{remark}
 
 \subsection{Quotients by quadratic elements}
 
 In this section we make more careful use of the computations in the previous subsection to show that if $f $ is a nonzero homogeneous element of $\S^2(W_+)$ and $J$ is a Poisson ideal of $\S(W_+)$ containing $f$, then $\dim \left( \S(W_+)/J \right)_n$ has polynomial growth, see Proposition~\ref{prop:part3}.  
 It follows that if $p$ is any order 2 element of $\U(W_+)$, then $\GK \U(W_+)/(p) < \infty$, see Corollary~\ref{cor:GKsofar}.
 
 We begin by establishing notation.  
 Recall the terminology of Definition~\ref{def:growth}.  
 If $f:  \NN \to \RR^+$ is a function with $\sG(f) \leq \sP(d)$ for some $d \in \NN$, we say that 
  $f(n) = O(n^d)$ and that $f$ has {\em polynomial growth}.
 
 For $k \in \NN$, 
  let $P_k(n)$ be the number of partitions of $n$ in which all parts are size $\leq k$.
  Recall that by \cite[Corollary~1.4.3.10]{Stanley}, 
 \beq\label{part} P_k(n) = O(n^{k-1}) . \eeq

 Given  $k, \ell \in \NN$ with $1 \leq k \leq \ell$, let $J(k,\ell)$ be the ideal
 of $\S(W_+)$ generated by $\{ x_i x_j : i \geq k, j-i \geq \ell -k\}$.
 
 \begin{lemma}\label{lem:part2}
 Let  $k, \ell \in \NN$ with $1 \leq k \leq \ell$.
 Then $\dim ( \S(W_+)/J(k,\ell))_n = O(n^{\ell -1})$.
 \end{lemma}
 \begin{proof}
 Since $J(k, \ell)$ is a monomial ideal, it suffices to count the monomials not in $J(k, \ell)$.
 Write a monomial in $\S(W_+)$ as $x_{\blambda} = x_{\lambda_1} \cdots x_{\lambda_d}$, where $\blambda =(\lambda_1 \leq \dots \leq\lambda_d)$ is a partition.
 Further, if $\lambda_{d-1} \geq k$, set $$e= e(\blambda) = \min (j: \lambda_j \geq k).$$
 
 There are three ways to have $x_{\blambda} \not \in J(k, \ell)$.
 Let
 \[ f_1(n) = \# \{ \blambda \vdash n : \lambda_d < k \},\]
 \[ f_2(n) = \# \{ \blambda \vdash n:  \lambda_d \geq k, \lambda_{d-1} < k \},\]
 \[ f_3(n) = \# \{ \blambda \vdash n: \lambda_{d-1} \geq k, \lambda_d - \lambda_e < \ell - k.\}\]
 Then $$\dim (\S(W_+)/J(k, \ell))_n = f_1(n) + f_2(n) + f_3(n),$$ so to prove the result we must estimate the growth of $f_1, f_2, f_3$.
 
 Clearly $f_1(n) = P_{k-1}(n) = O(n^{k-2})$ by \eqref{part}.
 We have 
 \[ f_2(n) = \sum_{m < n-k} P_{k-1}(m) = O(n^{k-1}),\]
 again as a consequence of \eqref{part}.
 Finally, if $\ell=k$ then $f_3(n) = 0$, so we may assume that $\ell > k$.
 Then partitions counted by $f_3$ involve, for some $b \geq k$, only the numbers $$1, 2, \dots, k-1, b, b+1, \dots, b+\ell-k-1.$$
 Thus $f_3(n) $ is less than or equal to the number of ways to write
 \[ n = \sum_{i=1}^{k-1} a_i i + \sum_{j=0}^{\ell -k -1} b_j(b+j),\]
 for $a_i, b_j \geq 0, b \geq k$.  
 
 In the equation above $a_1$ is determined by $$a_2, \dots, a_{k-1}, b_0, b_{\ell-k-1}, b.$$  
 As each of the $a_i, b_j, b \leq n$, we have that $f_3(n) \leq n^{\ell-1}$.
 This proves the result.
 \end{proof}
 
 \begin{lemma}\label{lem:kl}
 Let $f$ be a nonzero element of $\S^2(W_+)$, and let $M = \U(W_+) \cdot f$.  There is some $\gamma_0  \in \Gamma$ so that $\gamma(M) \supseteq \gamma_0 + \Gamma$.
 \end{lemma} 
 \begin{proof}
 As in the proof of Theorem~\ref{thm:S2}, let $\Sigma$ be the sub-semigroup of $\NN\times \NN$ generated by $\{(0,1), (3,3)\}$.  
 Since by the proof of Theorem~\ref{thm:S2} $\gamma(M)$ is a $\Sigma$-subrepresentation of $\Gamma$, the elements 
  $$\gamma(f) +(0,1), \gamma(f)+(1,5), \gamma(f)+(2,4), \gamma(f)+(3,3) $$
  are in $\gamma(M)$.  
  Thus 
   \begin{multline*}
 \gamma(M) \supseteq (\gamma(f)+(0,4) + \Sigma) \cup (\gamma(f)+(1,5) + \Sigma) \cup (\gamma(f)+(2,6)+ \Sigma) \\
 = \gamma(f)+(0,4) + \bigl(\Sigma \cup ( (1,1) + \Sigma ) \cup 
 ( (2,2)+ \Sigma)\bigr) = \gamma(f)+(0,4) + \Gamma.
 \end{multline*}
 Thus we may take $\gamma_0 = \gamma(f) + (0,4)$.
 \end{proof}

 \begin{proposition}\label{prop:part3}
 Let $J$ be a  Poisson ideal of $\S(W_+)$ that contains an element of order less than or equal to two.
 Then 
 \[\PGK(\S(W_+)/J) < \infty.\]
 \end{proposition}
 \begin{proof}
 If $J$ contains an element of order $1$ the result is implied by Lemma~\ref{Lalev}.  For the order 2 case, we will use Lemma~\ref{lem:kl}.
 
 Let $\blambda = (\lambda_1 \leq   \dots \leq \lambda_k)$ be a partition, and write 
 $x_{\blambda} = x_{\lambda_1} \dots x_{\lambda_k}$ as \[ x_1^{m_1} x_2^{m_2} \dots x_{\lambda_k}^{m_{\lambda_k}}\] for some $m_1, \dots, m_k \in \NN$.
  Define $$\underline {m}(\blambda)  = (m_1, \dots, m_{\lambda_k}, 0, \dots) \in \NN^{\oplus \NN}.$$ 
  
 Let $\prec$ be the graded reverse lexicographic order on partitions.
 That is, if $\blambda, \bmu$ are partitions we say $\blambda \prec \bmu$ if either $\abs{\blambda} < \abs{\bmu}$ or $\abs{\blambda} = \abs{\bmu}$ and the rightmost nonzero entry of $\underline{m}(\bmu)-\underline{m}(\blambda)$ is positive.
 (Alternately, if $\abs{\blambda} = \abs{\bmu}$ then $\blambda < \bmu$ if and only if the reversed sequence $\blambda^{\op} = (\lambda_k, \dots, \lambda_1)$ precedes $\bmu^{\op}$ in lexicographic order.)
 Note that $\prec$ generalises the order defined previously on 
  $$\Gamma = \{ (i,j) \in \ZZ^2 : 1 \leq i \leq j\},$$ the grading semigroup for $\S^2(W_+)$.
  
  We also define $\prec$ on monomials in $\S(W_+)$ by saying $x_{\blambda} \prec x_{\bmu}$ if $\blambda \prec \bmu$.
  Note that $\prec$ is a monomial ordering in the sense of \cite[Definition 2.2.1]{CLO}:
  that is, if $\blambda, \bmu, \bnu$ are partitions and $x_{\blambda} \prec x_{\bmu}$, then $$x_{\blambda} x_{\bnu} \prec x_{\bmu} x_{\bnu}.$$
  If $f \in \S(W_+)$, write $\LT(f)$ for the largest monomial in $f$ in the $\prec$ ordering.
 
Since $J \cap \S^2(W_+)$ is a $\U(W_+)$-submodule of $\S^2(W_+)$, by Lemma~\ref{lem:kl}  we have $(k + \ell)+\Gamma \subseteq \gamma(J\cap \S^2(W_+))$ for some $(k, \ell) \in \Gamma$. 
Thus for $i \geq k$, $j-i \geq \ell -k$, there is some $$f_{ij} \in J \cap \S^{\leq 2}(W_+)$$ with $\LT(f_{ij}) = x_ix_j$.
 Since  $\prec$ is a monomial ordering,  for all partitions  $\blambda $ with  $x_{\blambda}\in J(k, \ell)$ there is some $f_{\blambda} \in J$ with  $\LT(f_{\blambda}) = x_{\blambda}$.
 Thus for any $g \in \S(W_+)$, by successively subtracting scalar multiples of the $f_{\blambda}$ we see that there is $g' \in \S(W_+)$ so that $g -g' \in J$ and so that $g'$ is a sum of monomials not in $J(k,\ell)$; further, $d(g) \leq d(f)$.
 
 For a Poisson ideal $K$ of $\S(W_+)$ define $$(\S(W_+)/K)_{\leq n} = \S(W_+)_{\leq n}/(\S(W_+)_{\leq n} \cap K).$$  This is a discrete, finite, exhaustive filtration on $\S(W_+)/K$.
 By Lemma~\ref{lem:PGKfilter} and Remark~\ref{rem:altfilter}, $$\PGK \S(W_+)/J \leq \varlimsup \log_n \dim (\S(W_+)/J)_{\leq n}.$$
 From the previous paragraph, 
  $$\dim (\S(W_+)/J)_{\leq n} \leq \dim (\S(W_+)/J(k,\ell))_{\leq n} = O(n^{\ell})$$ by Lemma~\ref{lem:part2}.
Thus  $$\varlimsup \log_n \dim (\S(W_+)/J)_{\leq n} \leq \varlimsup \log_n \dim (\S(W_+)/J(k , \ell))_{\leq n} \leq \ell.$$
  \end{proof}
 
 \begin{corollary}\label{cor:GKsofar}
  If $f $ is a nonzero element of $ \U(W_+)$ with $o(f) \leq 2$, then $$\GK \U(W_+)/(f) < \infty.$$
 \end{corollary}
 \begin{proof}
This follows from Theorem~\ref{thm:GKW2} and Proposition~\ref{prop:part3}.
\end{proof}

We conjecture that Corollary~\ref{cor:GKsofar} is true without restriction on $o(f)$, see Conjecture~\ref{conj2}. 
Likewise, we conjecture that Proposition~\ref{prop:part3} holds for arbitrary Poisson ideals of $\S(W_+)$.

Recall that a module $M$ is {\em GK $d$-critical} if $\GK M = d$ and the GK-dimension of any proper quotient of $M$ is $< d$. That the adjoint representation of $W_+$ is GK 1-critical is Lemma~\ref{Lalev}.

\begin{corollary}\label{cor:2crit}
As a $\U(W_+)$-module, $\S^2(W_+)$ is GK 2-critical.
\end{corollary}
\begin{proof}
Let $f \neq 0$ be an element of $\S^2(W_+)$ and let   $M = \U(W_+)\cdot f$. 
Let  $\gamma = (k, \ell)$ be the element of $\Gamma$ given by Lemma~\ref{lem:kl}, so $\gamma(M) \supseteq (k, \ell) + \Gamma$.
Let $N = \S^2(W_+)/M$ and for $d \in \NN$ let $N_{\leq d}$ be the image of $\S^2(W_+)_{\leq d}$ in $N$.  

As in the proof of Proposition~\ref{prop:part3}, for any $g \in \S^2(W_+)$ there is $g' \in\S^2(W_+)$ so that $g - g' \in M$, $d(g') \leq d(g)$, and $g'$ involves only monomials of the form $x_i x_j$ with $i \leq k $ or $j \leq \ell$.
For fixed $d$, the number of such $x_i x_j$ with $i+j = d$ is $\leq k+\ell$.
Thus $\dim N_{\leq d} \leq (k+\ell)d$ and so grows at most linearly in $d$, and it follows that $\GK N \leq 1$ as desired.

It is easy to check that  $\GK \S^2(W_+)=2$ 
 by similar the arguments to those in  the proof of Lemma~\ref{lem:kl}.
We leave the details to the reader.
\end{proof}

\section{Appendix:  Macaulay2 computations} \label{APPENDIX}
We present the routines needed for the proof of Lemma~\ref{lem:S2one}.

\noindent {\bf Routine A.1.} The following Macaulay2 code is used in the proof of the general case of Lemma~\ref{lem:S2one}.

We first define the vectors $v_0, \dots v_6$ from the proof of Lemma~\ref{lem:S2one}.

{\footnotesize
\begin{verbatim}
Macaulay2, version 1.7
with packages: ConwayPolynomials, Elimination, IntegralClosure, LLLBases, PrimaryDecomposition, ReesAlgebra, 
         TangentCone
\end{verbatim}
}
\vspace{-.15in}
\begin{multicols}{2}
{\footnotesize
\begin{verbatim}
i1 : R = QQ[I,J];
i2 : v0={(J-6),	      	   	     	       
     	  (J+4)*(J-5),                 
     	  (J+2)*(J-4),     	    
     	  (J+4)*(J+3)*(J-4),         	 	 
     	  J*(J-3),	               
     	  (J+4)*(J+1)*(J-3),            
     	  (J+4)*(J+3)*(J+2)*(J-3),      
     	  (J+2)*J*(J-2),    	        
     	  (J+4)*(J+3)*J*(J-2),	       	 
     	  (J+4)*(J+3)*(J+2)*(J+1)*(J-2),    
     	  (J+4)*(J+3)*(J+2)*(J+1)*J*(J-1)}; 
i3 : v1={0, (I-1)*(J-5),                 
     	  0, 2*(I-1)*(J+3)*(J-4),          	 
     	  0, (I-1)*(J+1)*(J-3),             
     	  3*(I-1)*(J+3)*(J+2)*(J-3),      
     	  0, 2*(I-1)*(J+3)*J*(J-2),	       	  
     	  4*(I-1)*(J+3)*(J+2)*(J+1)*(J-2),   
     	  6*(I-1)*(J+3)*(J+2)*(J+1)*J*(J-1)};   	  
i4 : v2={0, 0, (I-2)*(J-4),     	   	   
     	  I*(I-1)*(J-4),         	 	 
     	  0, (I-2)*(J+2)*(J-3),             
     	  3*I*(I-1)*(J+2)*(J-3),      
     	  3*(I-2)*J*(J-2),    	      
     	  I*(I-1)*J*(J-2)+2*(I-2)*(J+2)*(J+1)*(J-2),	       	  
     	  6*I*(I-1)*(J+2)*(J+1)*(J-2)
     	   +(I-2)*(J+2)*(J+1)*J*(J-1),   
     	  15*I*(I-1)*(J+2)*(J+1)*J*(J-1)};  
i5 : v3={0, 0, 0, 0, 2*(I-3)*(J-3),	                
     	  (I+1)*(I-2)*(J-3)+(I-3)*(J+1)*(J-2),         
     	  (I+1)*I*(I-1)*(J-3)+(I-3)*(J+1)*J*(J-1),       
     	  0, 4*(I+1)*(I-2)*(J+1)*(J-2),	       	  
     	  4*(I+1)*I*(I-1)*(J+1)*(J-2)
     	   +4*(I+1)*(I-2)*(J+1)*J*(J-1),   
     	  20*(I+1)*I*(I-1)*(J+1)*J*(J-1)};  
i6 : v4={0, 0, (I-4)*(J-2),     	   	   
     	  (I-4)*J*(J-1),         	 	 
     	  0, (I+2)*(I-3)*(J-2),            
     	  3*(I+2)*(I-3)*J*(J-1),     
     	  3*I*(I-2)*(J-2),    	      	   
     	  2*(I+2)*(I+1)*(I-2)*(J-2)+I*(I-2)*J*(J-1),	        
     	  (I+2)*(I+1)*I*(I-1)*(J-2)
     	   +6*(I+2)*(I+1)*(I-2)*J*(J-1),    
     	  15*(I+2)*(I+1)*I*(I-1)*J*(J-1)};  
i7 : v5={0, (I-5)*(J-1),                   
     	  0, 2*(I+3)*(I-4)*(J-1),         	 	 
     	  0, (I+1)*(I-3)*(J-1),             
     	  3*(I+3)*(I+2)*(I-3)*(J-1),     
     	  0, 2*(I+3)*I*(I-2)*(J-1),	       	  
     	  4*(I+3)*(I+2)*(I+1)*(I-2)*(J-1), 
     	  6*(I+3)*(I+2)*(I+1)*I*(I-1)*(J-1)};      
i8 : v6={(I-6),	      	   	     	       
     	  (I+4)*(I-5),                  
     	  (I+2)*(I-4),     	   	   	 
     	  (I+4)*(I+3)*(I-4),         	  
     	  I*(I-3),	               
     	  (I+4)*(I+1)*(I-3),             
     	  (I+4)*(I+3)*(I+2)*(I-3),       
     	  (I+2)*I*(I-2),    	      	   
     	  (I+4)*(I+3)*I*(I-2),	       	  
     	  (I+4)*(I+3)*(I+2)*(I+1)*(I-2),   
     	  (I+4)*(I+3)*(I+2)*(I+1)*I*(I-1)};  
\end{verbatim}
}
\end{multicols}	
We define an automorphism $f$ of $ \ZZ[i,j]$ which sends $i \mapsto i+1, j \mapsto j-1$.
\begin{multicols}{2}
{\footnotesize
\begin{verbatim}       
i9 : f=map(R,R,{I+1,J-1});
i10 : dof = L -> toList apply(0..10, i->f(L#i));
\end{verbatim}
}
\end{multicols}	
We compute the locus on which the vectors $v_0, v_1, v_2, v_3, f(v_0), f(v_1), f(v_2), f^2(v_0), f^2(v_1), f^3(v_0)$ are linearly independent, and find the top-dimensional components of this locus.
\begin{multicols}{2}
{\footnotesize
\begin{verbatim}     
i11 : M=matrix{v0,v1,v2,v3,dof(v0),dof(v1),
          dof(v2), dof(dof(v0)),dof(dof(v1)),
          dof(dof(dof(v0)))};
i12 : K=minors(10,M);
i13 : KK=topComponents K;
i14 : associatedPrimes KK
o14 = {ideal I, ideal(J - 1), ideal(J + 1), 
          ideal(I + 1), ideal(I - 1), ideal(I - J + 3)}
\end{verbatim}
}
\end{multicols}

\noindent {\bf Routine A.2.} 
This routine is used for the case $j=i+6$ of Lemma~\ref{lem:S2one}.
\vspace{-.15in}

\begin{multicols}{2}
{\footnotesize
\begin{verbatim}
i15 : S=QQ[I];
i16 : g6=map(S,R,{I,I+6});
i17 : dg6 = L -> toList apply(0..10, i->g6(L#i));
i18 : N6=matrix{dg6(v0),dg6(v1),dg6(v2),dg6(v3), 
        dg6(dof(v0)), dg6(dof(v1)),dg6(dof(v2)), 
       dg6(dof(dof(v0))), dg6(dof(dof(v1))),
       dg6(dof(dof(dof(v0+v6))))};
i19 : J6=minors(10,N6);
i20 : associatedPrimes J6
o20 = {ideal I, ideal(I - 1), ideal(I + 5), 
         ideal(I + 1), ideal(I + 7)}
  \end{verbatim}
}
\end{multicols}

\noindent {\bf Routine A.3.} 
This routine is used for the case $j=i+5$ of Lemma~\ref{lem:S2one}.

\vspace{-.15in}
\begin{multicols}{2}
{\footnotesize
\begin{verbatim}
i21 : g5=map(S,R,{I,I+5});
i22 : dg5 = L -> toList apply(0..10, i->g5(L#i));
i23 : N5=matrix{dg5(v0),dg5(v1),dg5(v2),dg5(v3),
          dg5(dof(v0)),dg5(dof(v1)),dg5(dof(v2)),
          dg5(dof(dof(v0))),dg5(dof(dof(v1+v6)))};
i24 : J5=minors(9,N5);
i25 : associatedPrimes J5
o25 = {ideal I, ideal(I - 1)}
  \end{verbatim}
}
\end{multicols}
\noindent {\bf Routine A.4.} 
This routine is used for the case $j=i+4$ of Lemma~\ref{lem:S2one}.

\vspace{-.15in}

\begin{multicols}{2}
{\footnotesize
\begin{verbatim}
i26 : g4=map(S,R,{I,I+4});
i27 : dg4 = L -> toList apply(0..10, i->g4(L#i));
i28 : N4=matrix{dg4(v0),dg4(v1),dg4(v2),dg4(v3),
          dg4(dof(v0)),dg4(dof(v1)),dg4(dof(v2+v6)),
          dg4(dof(dof(v0+v6))),dg4(dof(dof(v1+v5)))};
i29 : J4=minors(9,N4);
i30 : associatedPrimes J4
o30 = {ideal I, ideal(I - 1), ideal(I + 1)}
  \end{verbatim}
}
\end{multicols}
\noindent {\bf Routine A.5.} 
This routine is used for the case $j=i+3$ of Lemma~\ref{lem:S2one}.
\vspace{-.15in}

\begin{multicols}{2}
{\footnotesize
\begin{verbatim}
i31 : g3=map(S,R,{I,I+3});
i32 : dg3 = L -> toList apply(0..10, i->g3(L#i));
i33 : N3=matrix{dg3(v0),dg3(v1),dg3(v2),dg3(v3+v6),
        dg3(dof(v0)),dg3(dof(v1+v6)),dg3(dof(v2+v5))};
i34 : J3=minors(7,N3);
i35 : associatedPrimes J3
o35 = {ideal(I - 1), ideal(2I + 3)}
  \end{verbatim}
}
\end{multicols}
\noindent {\bf Routine A.6.} 
This routine is used for the case $j=i+2$ of Lemma~\ref{lem:S2one}.
\vspace{-.15in}

\begin{multicols}{2}
{\footnotesize
\begin{verbatim}
i36 : g2=map(S,R,{I,I+2});
i37 : dg2 = L -> toList apply(0..10, i->g2(L#i));
i38 : N2=matrix{dg2(v0),dg2(v1),dg2(v2+v6),dg2(v3+v5),
        dg2(dof(v0+v6)),dg2(dof(v1+v5)),
        dg2(dof(v2+v4))};
i39 : J2=minors(7,N2);
i40 : associatedPrimes J2
o40 = {ideal I, ideal(I - 1), ideal(I + 1)}
  \end{verbatim}
}
\end{multicols}
\noindent {\bf Routine A.7.} 
This routine is used for the case $j=i+1$ of Lemma~\ref{lem:S2one}.
\vspace{-.15in}

\begin{multicols}{2}
{\footnotesize
\begin{verbatim}
i41 : g1=map(S,R,{I,I+1});
i42 : dg1 = L -> toList apply(0..10, i->g1(L#i));
i43 : N1=matrix{dg1(v0),dg1(v1+v6),dg1(v2+v5),
        dg1(v3+v4)};
i44 : J1=minors(4,N1);
i45 : associatedPrimes J1
o45 = {}
  \end{verbatim}
}
\end{multicols}
\noindent {\bf Routine A.8.} 
This routine is used for the case $j=i$ of Lemma~\ref{lem:S2one}.
\vspace{-.15in}

\begin{multicols}{2}
{\footnotesize
\begin{verbatim}
i46 : g0=map(S,R,{I,I});
i47 : dg0 = L -> toList apply(0..10, i->g0(L#i));
i48 : N0=matrix{dg0(v0+v6),dg0(v1+v5),
        dg0(v2+v4), dg0(v3)};
i49 : J0=minors(4,N0);
i50 : associatedPrimes J0
o50 = {ideal(I - 1)}
  \end{verbatim}
}
\end{multicols}


\begin{thebibliography}{SW2}

\bibitem[B]{Bell}  J. P. Bell. Examples in finite Gel'fand-Kirillov dimension. II. {\em Comm. Algebra} {\bf 33} (2005), 3323--3334.

\bibitem[CLO]{CLO} D. Cox, J. Little, and D.  O'Shea. {\em Ideals, varieties, and algorithms. An introduction to computational algebraic geometry and commutative algebra}. Undergraduate Texts in Mathematics. Springer-Verlag, New York, 1992. 

\bibitem[Dix]{Dix} J. Dixmier. {\em Enveloping algebras}. Graduate Studies in Mathematics,  11. American Mathematical Society, Providence, RI, 1996.

\bibitem[Eis]{Eisenbud}  D. Eisenbud. {\em Commutative algebra with a view toward algebraic geometry}. Graduate Texts in Mathematics, 150. Springer-Verlag, New York, 1995. 

\bibitem[Kap]{DiffAlg} I. Kaplansky. {\em An introduction to differential algebra. } Actualites Scientifiques et Industrielles, No. 1251. Publications de l'Institut de Math\'ematique de l'Universit\'e de Nancago, No. V. Hermann, Paris, 1957.

\bibitem[KL]{KL} G. R. Krause and T. H. Lenagan. {\em Growth of algebras and Gelfand-Kirillov dimension}.  Revised edition. Graduate Studies in Mathematics, 22. American Mathematical Society, Providence, RI, 2000.

\bibitem[Mar]{Marker} D. Marker. Chapter 2: Model Theory of Differential Fields, in {\em Model Theory of Fields}, 38--113. Springer-Verlag, Berlin, 1996.

\bibitem[MR]{MR}  J. C. McConnell and J. C.  Robson.   {\em Noncommutative Noetherian rings}. Revised edition. Graduate Studies in Mathematics, 30. American Mathematical Society, Providence, RI, 2001.

\bibitem[PP1]{PP1} I. Penkov and A. Petukhov. On ideals in the enveloping algebra of a locally simple Lie algebra. {\em Int. Math. Res. Not. IMRN}  (2015), no. 13, 5196--5228.  

\bibitem[PP2]{PP2} I. Penkov and  A. Petukhov. Primitive ideals of $\U(\mathfrak{sl}(\infty))$. https://arxiv.org/pdf/1608.08934.pdf.


\bibitem[SW1]{SW1} S. J. Sierra and C. Walton.
The universal enveloping algebra of the Witt algebra is not noetherian.
{\em Adv. Math.} {\bf 262} (2014), 239-260. 

\bibitem[SW2]{SW2} S. J. Sierra and C. Walton. Maps from the enveloping algebra of the positive Witt algebra to regular algebras.
{\em Pacific J. Math} {\bf 284} (2016), no. 2, 475--509.

\bibitem[Smi]{Smith} S. P. Smith. Gelfand-Kirillov dimension of rings of formal differential operators on affine varieties.
{\em Proc. Amer. Math. Soc.} {\bf 90 } (1984), no. 1, 1--8. 

\bibitem[Sta]{Stanley} R.  P. Stanley. {\em Enumerative combinatorics.} Volume 1. Second edition. Cambridge Studies in Advanced Mathematics, 49. Cambridge University Press, Cambridge, 2012

\end{thebibliography}
\end{document}